\documentclass[11pt,reqno]{amsart}
\usepackage{amsmath}
\usepackage{cases}
\usepackage{amssymb, latexsym, amsmath}
\usepackage[colorlinks,
linkcolor=red,
anchorcolor=magenta,
citecolor=blue]{hyperref}
\usepackage{hyperref}
\usepackage{mathrsfs}
\usepackage{amscd}
\usepackage{mathtools}
\usepackage{amsbsy}
\usepackage{graphicx}
\usepackage{color}
\usepackage[all]{xy}
\usepackage{bigstrut}
\usepackage{geometry}
\begin{document}
\newtheorem{theoreme}{Theorem}
\newtheorem{lemma}{Lemma}[section]
\newtheorem{proposition}[lemma]{Proposition}
\newtheorem{corollary}[lemma]{Corollary}
\newtheorem{definition}[lemma]{Definition}
\newtheorem{conjecture}[lemma]{Conjecture}
\newtheorem{remark}[lemma]{Remark}
\newtheorem{exe}{Exercise}
\newtheorem{theorem}[lemma]{Theorem}
\theoremstyle{definition}
\numberwithin{equation}{section}
\newcommand{\R}{\mathbb R}
\newcommand{\TT}{\mathbb T}
\newcommand{\Z}{\mathbb Z}
\newcommand{\N}{\mathbb N}
\newcommand{\Q}{\mathbb Q}
\newcommand{\Var}{\operatorname{Var}}
\newcommand{\tr}{\operatorname{tr}}
\newcommand{\supp}{\operatorname{Supp}}
\newcommand{\intinf}{\int_{-\infty}^\infty}
\newcommand{\me}{\mathrm{e}}
\newcommand{\mi}{\mathrm{i}}
\newcommand{\dif}{\mathrm{d}}
\newcommand{\beq}{\begin{equation}}
\newcommand{\eeq}{\end{equation}}
\newcommand{\beqq}{\begin{equation*}}
\newcommand{\eeqq}{\end{equation*}}
\newcommand{\ben}{\begin{eqnarray}}
\newcommand{\een}{\end{eqnarray}}
\newcommand{\beno}{\begin{eqnarray*}}
\newcommand{\eeno}{\end{eqnarray*}}

\def\d{\delta}
\def\a{\alpha}
\def\e{\varepsilon}
\def\ld{\lambda}
\def\p{\partial}
\def\v{\varphi}
\newcommand{\D}{\Delta}
\newcommand{\Ld}{\Lambda}
\newcommand{\n}{\nabla}
\newcommand{\GG}{\text{g}}
\newcommand{\f}{\frac}
\newcommand{\dec}{L^{p,R}_{\rm Dec}}

\title[Local smoothing of FIO]
{Square function inequality for  a class of Fourier integral operators satisfying cinematic curvature conditions }

\author[Gao]{Chuanwei Gao}
\address{\hskip-1.15em Chuanwei Gao:
	\hfill\newline  Beijing International Center for Mathematical Research,
	\hfill\newline Peking University, Beijing 100871, China}
\email{cwgao@pku.edu.cn}

\author[Miao]{Changxing Miao}
\address{\hskip-1.15em Changxing Miao:
	\hfill\newline Institute of Applied Physics and Computational
	Mathematics,
	\hfill\newline P. O. Box 8009,\ Beijing,\ China,\ 100088,}
\email{miao\_changxing@iapcm.ac.cn}

\author[Yang]{Jianwei-Urbain Yang}
\address{\hskip-1.15em Jianwei-Urbain Yang£º
	\hfill\newline Department of Mathematics,
	\hfill\newline Beijing Institute of Technology,
	\hfill\newline Beijing 100081,\ P. R.  China}
\email{jw-urbain.yang@bit.edu.cn}


\begin{abstract}
In this paper, we establish an improved variable coefficient version of square function inequality, by which  the local smoothing  estimate $L^p_\alpha\rightarrow L^p$ for the Fourier integral operators satisfying cinematic curvature condition is further improved.  In particular, we establish almost sharp results for $2<p\leq 3$ and push forward the estimate for the critical point $p=4$. As a consequence, the local smoothing estimate for the wave equation on the manifold is refined. We generalize the results in \cite{LeVa12, Le18P} to its variable coefficient counterpart.
	The main ingredients in the argument includes  multilinear oscillatory integral  estimate \cite{BCT06} and decoupling inequality \cite{BelHicSog18P}.
\end{abstract}
\maketitle

\section{Introduction}
\label{sect:introd}

In this paper, we  are concerned with a class of Fourier integral operator $\mathscr{F}\in I^{\sigma-\frac{1}{4}}(Z,Y;\mathscr{C})$ of which the canonical relation $\mathscr{C}$ satisfies a certain curvature condition which is usually referred to as cinematic condition.  The typical example, among other things, is the half-wave operator $e^{it\sqrt{-\Delta}}$.   This class of operators are naturally related to the solution of wave equation on general Riemannian manifolds.  Local smoothing conjecture due to Sogge \cite{Sogge91} is about the sharp regularity $L_\alpha^p\rightarrow L^p$ estimate for the half-wave operator $e^{it\sqrt{-\Delta}}$ and  was found many applications in harmonic analysis. It roughly says that one may gain additional regularity by averaging time variable over the fixed time case. It is natural to consider this phenomenon at the general level of Fourier integral operator.  We refer to
\cite{BelHicSog18P,BelHicSog18P-surv,GaMiaYa18P} for more information about the motivation and backgrounds.

 Let $Z$ and $Y$ be  paracompact manifolds with
${\rm dim} \;Z=3$
and ${\rm dim}\; Y=2$.
We consider Fourier integral operators  $\mathscr{F}\in I^{\sigma-\frac{1}{4}}(Z,Y;\mathscr{C})$, where the canonical relation $\mathscr{C}$, from $T^*Y\setminus 0$ to $T^*Z\setminus 0$, which is a homogeneous, conic
Lagrangian submanifold of $T^*Z\setminus 0\times T^*Y\setminus 0$ with ${\rm dim}\; \mathscr{C}=5$.

As in \cite{MSS-jams},
we impose the cone conditions  on $\mathscr{C}$.
Given $z_0\in Z$, let $\varPi_{T^{*}Y}$, $\varPi_{T^{*}_{z_0}Z}$
and $\varPi_{Z}$ be projections from $\mathscr{C}$ to $T^{*}Y\setminus 0$, $T^*_{z_0}Z\setminus 0$ and $Z$ respectively,
\begin{equation}
\label{jsdkmkcsd}
\xymatrix{
	&\mathscr{C}\ar[d]_{\varPi_{Z}} \ar[dl]_{\varPi_{T^{*}Y}}\ar[dr]^{\varPi_{T^{*}_{z_0}Z}}	
	&\\ T^{*}Y\setminus 0
	&Z
	& T^{*}_{z_0}Z\setminus0}
\end{equation}
and assume
\begin{equation}
\label{dkmkldsc}
{\rm rank}\; d\varPi_{T^{*}Y}\equiv 4,
\end{equation}
\begin{equation}
\label{vmklsdmks}
{\rm rank}\; d \varPi_{Z}\equiv 3.
\end{equation}
Let $\Gamma_{z_0}=\varPi_{T^{*}_{z_0}Z}(\mathscr C)$.
As a consequence of \eqref{dkmkldsc}\eqref{vmklsdmks} and the  homogeneity,
$\Gamma_{z_0}$ is a conic subset of $T^*_{z_0}Z\setminus0$.
The \emph{cone condition} imposed on $\mathscr C$ is that for every $\zeta\in \Gamma_{z_0}$,
there is one principal curvature nonvanishing.
We say a Fourier integral operator $\mathscr F$ satisfies the cinematic curvature condition,
if its canonical relation $\mathscr C$ satisfies \eqref{dkmkldsc}\eqref{vmklsdmks} and the cone condition.

This paper is a continuation of our previous work \cite{GaMiaYa18P}. By adapting the strategy in \cite{LeVa12, Le18P}, we give further improvement upon our previous result in dimension two.
\begin{theoreme} \label{theo1}
	\label{main}
	Suppose  $\mathscr F\in I^{\sigma-\frac{1}{4}}(Z,Y;\mathscr C)$ with $Z, Y$  as above,
	where the canonical relation $\mathscr C$ satisfies the cinematic curvature condition.
	Then, the following estimate
	holds
	\begin{equation}
	\label{eq:main}
	\bigl\|\mathscr{ F }f\bigr\|_{L^p_{\rm loc}(Z)}\leqslant C\|f\|_{L^p_{\rm comp}(Y)}, \; \text{for all} \; \sigma<-\sigma(p),
	\end{equation}
	where $\sigma(p)$ is defined by
	\begin{equation}
	\label{eq:sig}
	\sigma(p)=\left\{ \begin{aligned}
	0\quad\;, &\quad2<p\leqslant 3,\\
	\frac{1}{4}-\frac{3}{4p}, &\quad 3<p\leqslant 4,\\
	\frac{3}{8}-\frac{5}{4p}, & \quad 4<p<6.
	\end{aligned}\right.
	\end{equation}
\end{theoreme}
\begin{remark}
  The results obtained above generalize its constant coefficient counterpart in \cite{LeVa12,Le18P}. It is worth noting that the results in the case $2<p\leq 3$ are sharp, except for possibly  arbitrarily small regularity loss.
  \end{remark}

As byproducts of Theorem \ref{theo1}, we have  a few consequences.

If we write $z=(x,t)$ and let $\mathscr{ F}_tf(x)=\mathscr Ff(x,t)$, then we have the following maximal theorem under the above assumptions in Theorem \ref{main}.
\begin{corollary}
	Let $I\subset\R$ be a compact interval and $Z=X\times I$ such that $X$ and $Y$ are assumed to be compact. Assume that $\mathscr F\in I^{\sigma-1/4}(Z,Y;\mathscr C)$ is a Fourier integral operator satisfying all the same conditions in Theorem \ref{main}.
	 Then we have
	\begin{equation}
	\label{eq:maximal}
	\|\sup_{t\in I}|\mathscr{F}_{t}f(x)|\|_{L^p(X)}\leqslant C\|f\|_{L^p(Y)},
	\end{equation}
	whenever $\sigma<-\sigma(p)-\frac1p$ for all $2\leqslant p\leqslant 6$.
\end{corollary}

Let $M$ be a smooth compact manifold without boundary of dimension $n$, equipped with a Riemmanian metric $\GG$ and consider
the Cauchy problem
\begin{equation}
\label{eq:0}
\left\{ \begin{aligned}
&(\partial/\partial t)^2-\Delta_\GG ) u(t,x)= 0,\,\;\;\;(t,x)\in \R\times M,\\
&u(0,x)=f(x),\;\;
\partial_tu(0,x)=h(x),\\
\end{aligned} \right.
\end{equation}
where $\Delta_\GG$ is the
Beltrami-Laplacian associated to a metric $\GG$. It is a well-known fact that the solution $u$ to this Cauchy problem can be written as
\begin{equation}
\label{eq:wave-solu}
u(x,t)=\mathscr{ F}_0f(x,t)+\mathscr{F}_1h(x,t),
\end{equation}
where $\mathscr{F}_j\in I^{j-1/4}(M\times \R,M;\mathscr C)$ with
\[	\mathscr{C}=\left\{(x,t,\xi,\tau,y,\eta):(x,\xi)=\chi_t(y,\eta),
\tau=\pm\sqrt{\sum\GG^{jk}\xi_j\xi_k}\right\},
\]
where $\chi_t:T^*M\setminus0\times T^*M\setminus 0$ is the flowing for time $t$ along the Hamilton vector field $H$ associated to $\sqrt{\sum\GG^{jk}\xi_j\xi_k}$.
As a consequence, the convexity condition is automatically verified by $\mathscr C$.
\begin{corollary}
	Let $u$ be the solution to the Cauchy problem
	\eqref{eq:0}.
	If $I\subset\R$ is a compact interval and $\sigma<-\sigma(p)$ with $\sigma(p)$ is given by \eqref{eq:sig}, then we have
	\begin{equation}
	\label{eq:wave}
	\|u\|_{L^p_{\alpha+\sigma}(M\times I)}
	\leqslant C
	\Bigl(
	\|f\|_{L^p_\alpha(M)}+\|h\|_{L^p_{\alpha-1}(M)}
	\Bigr), \;2\leqslant p\leqslant 6.
	\end{equation}
\end{corollary}

Let
$\varSigma_{x,t}\subset\R^2$
be a smooth curve depending smoothly
on the parameters
$(x,t)\in\R^2\times[1,2]$ and
$\dif \sigma_{x,t}$ denotes the
normalized Lebesgue
measure on $\varSigma_{x,t}$.
Following the notations in \cite{Sogge91,SchSog97}, we may assume
$\varSigma_{x,t}=\{y:\Phi(x,y)=t\}$
where $\Phi(x,y)\in C^\infty(\R^2\times\R^2)$ such that its Monge-Ampere determinant is non-singular
\begin{equation}
\label{eq:MA}
{\rm det}
\begin{pmatrix}
0 & \partial\Phi/\partial x  \\
\partial\Phi/\partial y & \partial^2\Phi/\partial x\partial y \\
\end{pmatrix}
\neq 0,\quad \text{when}\;
\Phi(x,y)=t.
\end{equation}
This  is referred to
as Stein-Phong's rotational
curvature condition.

Define the averaging operator by
\begin{equation}
\label{eq:A}
A f(x,t)
=\int_{\varSigma_{x,t}}
f(y)a(x,y)\dif \sigma_{x,t}(y),
\end{equation}
where $a(x,y)$ is a smooth function with compact support in $\R^2\times\R^2$.

If we let
$\mathscr{F}:f(x)\mapsto
Af(x,t)$, then $\mathscr F$
is a Fourier integral operator of order $-1/2$ with canonical relation
given by
\begin{equation}
\label{eq:cano}
\mathscr{C}=\big\{(x,t,\xi,\tau,y,\eta):(x,\xi)=\chi_t(y,\eta),\tau=q(x,t,\xi)\big\},
\end{equation}
where $\chi_t$ is a local symplectomorphism, and
the function
$q$ is homogeneous of degree one in $\xi$ and smooth away from $\xi=0$.
Moreover, the rotational curvature condition holds if and only if
\begin{equation}
\label{eq:qq}
\left\{
\begin{aligned}
q(x,\Phi(x,y),\Phi'_x(x,y))&\equiv 1\\
\text{corank}\;\; q''_{\xi\xi}&\equiv 1.\\
\end{aligned} \right.
\end{equation}
In particular, the cinematic curvature condition is fulfilled and one has the following result by Theorem \ref{main}.

\begin{corollary}
	Let $A_t$ be an averaging operator defined in \eqref{eq:A} with $\varSigma_{x,t}$ satisfying the geometric conditions described as above.
	Then there exists a constant $C$ depending on $p$ such that if $f\in L^p(\R^2)$, we have
	\begin{equation}
	\|Af\|_{L^p_{\gamma}(\R^{2+1})}\leqslant C\|f\|_{L^p(\R^2)},\quad 2\leqslant p\leqslant 6,
	\end{equation}
	for all $\gamma<-\sigma(p)+\frac{1}{2}$.
\end{corollary}
The paper is organized as follows.
In Section \ref{sect2}, we introduce the implements that will be used in the subsequent context. This section is similar to that in \cite{GaMiaYa18P}. The differences focus on the quantitative hypothesis on the phase function which is introduced in \cite{BelHicSog18P} and will play
a crucial role in the induction argument via multilinear oscillatory estimates of Bennett, Carbery and Tao \cite{BCT06}. In particular, we emphasize the role played among other things by the parabolic rescaling.
In Section \ref{sect:3}, we show how to deduce the square function estimate from
the corresponding multilinear version.
This observation on the relationship between linear and multilinear estimates is originated back to Bourgain and Guth \cite{BoGu11}, and then applied to many important breakthroughs such as Bourgain and Demeter's proof of  $\ell^2-$decoupling theorem \cite{BoDe2015},  Bourgain-Demeter-Guth's proof of the main conjecture of Vinogradov's mean value theorem \cite{BoDeGu16}, as well as the cone multiplier problem by Lee and Vargas  \cite{LeVa12}.
In Section \ref{sect:4}, we prove the tri-linear square function estimate by using
the multilinear oscillatory integral estimates of \cite{BCT06}.
In the last section, we use decoupling  inequality to bring in further improvement upon the $L^4$ estimate.

\subsection*{Notations}
If $a$ and $b$ are two positive quantities, we write $a\lesssim b$ when there exists a constant $C>0$ such that $a\leq C b$ where the constant will be clear from the context. When the constant depends on some other quantity $M$, we emphasize the dependence by writing $a\lesssim_M b$.
We will write $a\approx b$ when we have both $a\lesssim b$ and $b\lesssim a$. We will write $a\ll b$ (resp. $a\gg b$) if there exists a sufficiently  large constant
$C>0$ such that $Ca\leq  b$ (resp. $a\geq Cb$).
We adopt the notion of nature numbers $\mathbb{N}=\mathbb{Z}\cap [0,+\infty)$.
For $\lambda\gg 1$,
we use ${\rm RapDec}(\lambda)$ to mean a quantity rapidly decreasing in $\lambda$.
We use $a\lessapprox b$
to mean $a\lesssim_\varepsilon R^\varepsilon b$ for arbitrary $\varepsilon$.
\\
\\
Throughout this paper, $w_B$ is a weight which is essentially concentrated on a ball $B\subset \R^3$ centered at $c(B)$ with radius $r(B)$ and rapidly decaying  away from the ball.
\[
w_B(z)\lesssim \left(1+\frac{|z-c(B)|}{r(B)}\right)^{-N},\quad N\gg 1.
\]
With the weight, we define the weighted Lebesgue norm $\|.\|_{L^p(w_{B})}$ as follows
\beq
\|f\|_{L^p(w_{B})}:=\Bigl(\int_{\R^3}|f(z)|^p w_{B}(z)\dif z\Bigr)^{\frac{1}{p}}\quad 1\leq p<\infty.
\eeq

\section{Preliminaries and Reductions}\label{sect2}
Given a point $(z_0,\zeta_0,y_0,\eta_0)\in T^*Z\setminus 0\times T^*Y\setminus 0$,
there exists a sufficiently small local conic coordinate patch around it,
along with a smooth function $\phi(z,\eta)$
such that $\mathscr C$ is given by
\begin{equation}
\label{qmmkds}
\{
(z,\phi'_z(z,\eta),\phi'_\eta(z,\eta),\eta):
\eta \in(\R^2\setminus 0)\cap \varGamma_{\eta_0}
\}
\end{equation}
where $\varGamma_{\eta_0} $ denotes a conic neighborhood of $\eta_0$.

By splitting
$z=(x,t)\in\R^2\times\R$ into space-time variables, where
we put $z_0=\mathbf{0}$ without loss of generality,
any operator
$\mathscr{F}$ in the class $I^{\sigma-1/4}(Z,Y;\mathscr C)$
with $\mathscr C$ satisfying the cinematic curvature condition can be written in an appropriate local
coordinates as a finite sum of
oscillatory integrals
\[
\mathscr{F}f(x,t)=\int_{\R^n} e^{i\phi(x,t,\eta)}b(x,t,\eta) \widehat{f}(\eta)\,\dif\eta,
\]
where $b$ is a smooth symbol of order $\sigma$.  We may assume that the support of the  map $z\to b(z,\eta) $ is contained in a ball
$B(\mathbf{0},\varepsilon_0)$, with $\varepsilon_0>0$ being sufficiently small and $\eta\to b(z,\eta)$ is supported in a conic region $\mathcal{V}_{\varepsilon_0}$, i.e.
$$b(x,t,\eta)=0
\;\text{if}\; \eta\notin\mathcal{V}_{\varepsilon_0}
:=\{\xi=(\xi_1,\xi_2)\in\R^2\backslash 0:|\xi_1|\leqslant \varepsilon_0\, \xi_2\}.$$

Let us turn to the strategy of the proof.
By interpolation with the trivial $L^2\to L^2$ estimate and the sharp $L^6\to L^6$ local smoothing estimate of \cite{BelHicSog18P},
\begin{equation}
\label{eq:L6}
\bigl\|\mathscr{ F }f\bigr\|_{L^6_{\rm loc}(Z)}\leqslant C\|f\|_{L^6_{\rm comp}(Y)}, \quad
\text{for all}\quad\sigma<-\frac{1}{6},
\end{equation}
one may reduce
\eqref{eq:main} to
\beq \label{eq:a1}
\bigl\|\mathscr{ F }f\bigr\|_{L^{p}_{\rm loc}(Z)}\leqslant C\|f\|_{L^{p}_{\rm comp}(Y)}, \quad
\text{for all}\quad\sigma<-\sigma(p),\;
\eeq
with $p=3,4$ respectively.

Fix $\lambda\gg 1$ and $\beta\in C_c^\infty(\mathbb{R})$ which vanishes outside  the interval $(1/4, 2 )$  and equals one in $(1/2, 1)$. By standard Littlewood-Paley decomposition,
one may reduce \eqref{eq:a1} to
\beq \label{eq:34}
\|\mathscr{F}_{\lambda}f\|_{L^p(\R^3)}
\lessapprox \lambda^{\sigma}\|f\|_{L^p{(\R^2})}, \quad \sigma>\sigma(p)
\eeq
where
$\mathscr{F}_\lambda $ is an operator
\begin{equation}
\label{mklsmcd}
\mathscr{F}_\lambda f(x,t)=
\int
e^{i\phi(x,t,\eta)}
b^\lambda (x,t,\eta)\widehat{f}(\eta)\,\dif\eta
\end{equation}
and
\beq
b^\lambda(z,\eta)=b(z,\eta)\frac{1}{(1+|\eta|^2)^{\sigma/2}}\beta(\frac{\eta}{\lambda}).
\eeq

Assume  $1\leq R\leq \lambda, \; \mathcal{C}(\mathbf{e}_2,\varepsilon_0)
:=B(\mathbf{e}_2,\varepsilon_0)\cap\mathbb{S}^{1}$ and
make angular decomposition with respect to the $\eta$-variable by cutting
$\mathcal{C}(\mathbf{e}_2,\varepsilon_0)$ into
$N_R\approx R^{\frac{1}{2}}$ many sectors
$\{\theta_{\nu}: 1 \leqslant \nu \leqslant N_R\}$,
each  $\theta_{\nu}$  spreading an angle  $\approx_{\varepsilon_0} R^{-1/2}$.

Let $\{\chi_{\nu}(\eta)\}$ be a family of smooth  cutoff functions associated with the decomposition in the angular direction,
each of which is  homogeneous of degree $0$,
such that $\{\chi_\nu\}_{\nu}$ forms a partition  of unity on the unit circle
and then extended homogeneously to $\R^2\setminus 0$ such that
\begin{equation*}\left\{\begin{aligned}
&\sum_{0\leqslant \nu\leqslant N_R} \chi_{\nu}(\eta)\equiv 1,\;\;\forall \eta \in \mathbb{R}^2\setminus 0,\\
&|\partial^{\alpha} \chi_{\nu}(\eta)|\leqslant C_\alpha R^{\frac{|\alpha|}{2}},\;\; \forall \;\alpha \; \text{if}\; |\eta|=1.\end{aligned}\right.
\end{equation*}
Define
\begin{align}
\mathscr{F}_\lambda^\nu f(x,t)
=\int
e^{i\phi^\lambda(x,t,\eta)}
b_{\lambda}^\nu(x,t,\eta)f(\eta)\,\dif \eta,
\end{align}
where the rescaled phase function and amplitude read
\beq
\phi^\lambda(x,t,\eta):=\lambda\phi(x/\lambda,t/\lambda,\eta), \; b_{\lambda}^\nu(x,t,\eta)=\chi_\nu(\eta)b^\lambda(x/\lambda,t/\lambda,\eta).
\eeq

Following the  approach in \cite{MSS-jams},  \eqref{eq:34} is in turn  reduced to the following kind of square function estimate
\begin{align}\label{eq100}
	\|\mathscr{F}_{\lambda}f\|_{L^p(\R^{3})}\lesssim_{\varepsilon,\phi, L} & \lambda^{\sigma(p)+\varepsilon}  \Big\|\Big(\sum_{\nu}|\mathscr{F}_{\lambda}^\nu f|^2\Big)^{\frac12}\Big\|_{L^p(\R^{3})}+\lambda^{-L}\|f\|_{L^2(\R^2)}
	\end{align}
	where $L$ is a number which can be taken sufficiently large and $p=3,4,R=\lambda$.

 Let $a(z,  \eta)\in C_c^\infty(\R^3\times \R^2)$ with compact support contained in $B(0,\varepsilon_0)\times B({\bf e_2},\varepsilon_0)$.
Define

\begin{align} \label{eq:200a}
\mathscr{T}_\lambda f=\sum_\nu \mathscr{T}_\lambda^\nu f,\;\;
\mathscr{T}_\lambda^\nu f(x,t)
=\int
e^{i\phi^\lambda(x,t,\eta)}
a_{\lambda}^\nu(x,t,\eta)f(\eta)\,\dif \eta,
\end{align}
where the rescaled  amplitude reads
\beq
a_{\lambda}^\nu(x,t,\eta)=\chi_\nu(\eta)a(x/\lambda,t/\lambda,\eta).
\eeq

By scaling argument,  \eqref{eq100} is in turn  reduced to the following square function estimate
\begin{theoreme}\label{pro1}
	Let the operator $\mathscr{T}_{\lambda}$ be as in \eqref{eq:200a} and take $R=\lambda$.  Then we have
	\begin{align}
	\label{eq:151}
	\|\mathscr{T}_{\lambda}f\|_{L^3(\R^{3})}\lesssim_{\varepsilon,\phi, L} & \lambda^{\varepsilon}  \Big\|\Big(\sum_{\nu}|\mathscr{T}_{\lambda}^\nu f|^2\Big)^{\frac12}\Big\|_{L^3(\R^{3})}+\lambda^{-L}\|f\|_{L^2(\R^2)},\\
	\label{eq:152}
	\|\mathscr{T}_{\lambda}f\|_{L^4(\R^{3})}\lesssim_{\varepsilon, \phi, L} & \lambda^{\frac{1}{16}+\varepsilon} \Big\|\Big(\sum_{\nu}|\mathscr{T}_{\lambda}^\nu f|^2\Big)^{\frac12}\Big\|_{L^4(\R^{3})}+\lambda^{-L}\|f\|_{L^2(\R^2)},
	\end{align}
	where $L$ is a number which can be taken sufficiently large.
\end{theoreme}

\begin{remark}\begin{itemize}
	\item It is worth noting  that the spacial support of the amplitude  appearing in the right-hand side of \eqref{eq:151} and \eqref{eq:152} is slightly larger than that appearing in the left-hand side.
 \item  Very recently,  Guth-Wang-Zhang \cite{GWZ} established the sharp square function estimate in the Euclidean case in $2+1$ dimensions. As a result, the corresponding local smoothing conjecture is resolved. For the variable coefficient setting, the Kakeya compression phenomena will happen which leads to the difference in the numerology of local smoothing conjecture between the variable and constant coefficient settings in $n\geq 3$, see \cite{BelHicSog18P-surv,BelHicSog18P,GHI} for more details. This work provides additional methods and techniques toward handling the variable coefficient case which may help advance the research in this direction.
\end{itemize}
\end{remark}

\subsection{Normalization of the phase function}
For technical reasons, we  assume $a$ is of the form  $a(z,\eta)=a_1(z) a_2(\eta)$, where $$a_1\in C_c^\infty (B(0,\varepsilon_0)), \quad a_2\in C_c^\infty(B({\bf e_2},\varepsilon_0)).$$ The general cases may be reduced to this special one via  the following observation
\begin{align*}
\int_{\R^3}e^{i\phi(z,\eta)}a(z,\eta)f(\eta)\dif \eta=\int_{\R^3} e^{i(z,\xi)}
\Bigl(\int_{\R^2} e^{i\phi(z,\eta)}\psi(z)\widehat{a}(\xi, \eta)  f(\eta) \dif \eta\Bigr) \dif \xi,
\end{align*}
and that $\xi\mapsto\widehat{a}(\xi,\eta)$ is a Schwartz function,
where $\psi(z)$ is a compactly supported smooth function and equals $1$ on ${\rm supp}_z\; a$.

Moreover, we may reformulate
\eqref{dkmkldsc} \eqref{vmklsdmks} and the curvature condition  as
\begin{itemize}
	\item[$\mathbf{H}_1$] ${\rm rank} \;\partial_{z\eta}^2 \phi(z,\eta)=2$  for all $(z,\eta)\in  {\rm supp}\;a$.
	\item[$\mathbf{H}_2$]  Define the Gauss map $G: {\rm supp}\;a \rightarrow \mathbb{S}^{2}$ by $G(z,\eta):=\frac{G_0(z,\eta)}{|G_0(z,\eta)|}$ where
	\beq
	G_0(x,\eta):=\partial_{\eta_1}\partial_z\phi(z,\eta)\wedge \partial_{\eta_2}\partial_z\phi(z,\eta).
	\eeq
	The curvature condition
	\beq
	{\rm rank}\; \partial_{\eta\eta}^2\langle \partial_z\phi(z,\eta), G(z,\eta_0) \rangle|_{\eta=\eta_0}=1
	\eeq
	holds for all $(z,\eta_0)\in {\rm supp}\;a$.
\end{itemize}

Let $\Upsilon_{x,t}:\eta\rightarrow \partial_{x}\phi(x,t,\eta)$.
If $\varepsilon_0$ is taken sufficiently small,
$\Upsilon_{x,t}$ is a local diffeomorphism on $B(\mathbf{e}_2,\varepsilon_0)$.
If we denote by $\Psi_{x,t}(\xi)=\Upsilon_{x,t}^{-1}(\xi)$  the inverse map of $\Upsilon_{x,t}$,
then clearly
\beq \label{eq:44}
\partial_x \phi(x,t, \Psi_{x,t}(\xi))=\xi.
\eeq
Differentiating  \eqref{eq:44} with respect to $\xi$ on both sides yields
\beq \label{eq:50}
[\partial^2_{x,\eta}\phi](x,t,\Psi_{x,t}(\xi))\;\partial_{\xi}
\Psi_{x,t}(\xi)={\rm Id}.
\eeq
This manifests that
\begin{equation}
\label{eq:mmmm}
{\rm det}\,\partial_{\xi}\Psi_{x,t}(\xi)\neq 0,\;\forall(x,t)\in B(\mathbf{0},\varepsilon_0),\;\forall \xi\in\Upsilon_{x,t}(B(\mathbf{e}_2,\varepsilon_0)).
\end{equation}
 In order to facilitate certain kind of induction argument, it is also useful to assume the quantitative conditions on the phase function.
Assume ${\rm supp}\; a \subset Z\times \Xi$, where $Z\subset \R^3$ is a small neighborhood of $B(0,\varepsilon_0)$, $\Xi\subset \R^2$ is a small open sector around $\mathbf{e}_2$.
Let ${\bf A}=(A_1,A_2,A_3)\in [1,\infty)$.

We adopt the notation from \cite{BelHicSog18P} with slight modifications.
Datum $(\phi, a)$  is said to be of type {\bf A} if the following  quantitative properties are satisfied
\begin{itemize}
	\item[($\mathbf{H}_{\bf A}$)]
	\beq
	\phi(x,t,\eta)=\langle x,\eta\rangle
	+\frac{t}{2}\,
	\eta_1^2/\eta_2
	+\eta_2\,\mathcal{E}(x,t,\eta_1/\eta_2)
	\eeq
	where  $\mathcal{E}(x,t,s)$ obeys
	\beq
	\mathcal{E}(x,t,s)\leqslant c_{\rm{par}}A_1 \big((|x|+|t|)^2|s|^2+(|x|+|t|)|s|^3 \big).
	\eeq
\item [($\mathbf{ D}_{{\bf A}}$)] For some large integer $N\in  \mathbb{N}$, depending only on the fixed choice of $\varepsilon, L$ and $p$, one has
\beq
\|\partial_{\eta}^\beta \partial_z^\alpha \phi\|_{L^\infty(Z\times \Xi)}\leq c_{\rm par}A_2
\eeq
for all $(\alpha,\beta)\in \mathbb{N}^{3}\times \mathbb{N}^2$ with $|\alpha|=2$ and $1\leq |\beta|\leq N$.
\item[($\mathbf{M}_{{\bf A}}$)]${\rm dist}({\rm supp}\;a_1,   \R^3\backslash Z\times \Xi)\geq A_3/4$.
\end{itemize}

  \begin{remark}
	The condition $(\mathbf{H}_{\bf A})$  demonstrates that the phase function can be viewed as a minor  perturbation of the translation-invariant case
	$
	\langle x,\eta\rangle
	+\frac{t}{2}\,
	\eta_1^2/\eta_2
	$.
	This fact will be useful in the course of  verifying  the transversality condition.  The reason we impose the hypothesis  $ (\mathbf{D}_{\bf A})$ is to bound the higher order derivatives of the phase function in the approximation argument.  As above mentioned, the support of amplitude function may be slightly enlarged in the induction. We may partition the  support of the amplitude to maintain  the margin hypothesis $( \mathbf{M}_{\bf A})$.
\end{remark}

\subsection{Parabolic rescaling}
Let $1\leq  R\leq \lambda$ and  $B_R $ denote  a  ball   of radius  $R$. For convenience, we introduce the   $l^2$ decoupling norm and square function norm respectively as
\begin{align}
\label{eq:oioioi}\|\mathscr{T}_{\lambda}f\|_{L^{p,R}_{\rm Dec}(B_R)}:&=\Bigl(\sum_{\nu}\|\mathscr{T}_{\lambda}^\nu  f\|_{L^p(w_{B_R})}^2\Bigr)^{\frac12},\\
\label{eq:ereree}\|\mathscr{T}_{\lambda}f\|_{L^{p,R}_{\rm Sq}(B_R)}:&=\Bigl\| \bigl(\sum_{\nu} |\mathscr{T}_{\lambda}^\nu f|^2\bigr)^{\frac12}\Bigr\|_{L^p(w_{B_R})}.
\end{align}
Let  $\mathbf{S}_{\bf A}^{\sigma,\varepsilon}(\lambda,R)$ be  the infimum over all $C$ such that
\beq\label{eq:101}
\|\mathscr{T}_{\lambda}f\|_{L^p(B_R)}\leqslant  C R^{\sigma+\varepsilon}
\|\mathscr{T}_{\lambda}f\|_{L^{p,R}_{\rm Sq}(B_R)}+R^4 \Bigl(\frac{\lambda}{R}\Bigr)^{-2M}\|f\|_{L^2(\R^2)},\;\forall\,R\leqslant \lambda^{1-\frac{\varepsilon}{2}}
\eeq
holds for all type ${\bf A}$ data $(\phi, a)$.

It is easy to see that $\mathbf{S}_{\bf A}^{\sigma, \varepsilon}(\lambda,R)$ is always finite. Indeed, by H\"older's inequality.
\beq
\|\mathscr{T}_{\lambda}f\|_{L^p(B_R)}
=\Bigl\|\sum_{\nu} \mathscr{T}_{\lambda}^\nu f\Bigr\|_{L^p(w_{B_R})}\lesssim  R^{\frac{1}{4}}\|\mathscr{T}_\lambda f\|_{L^{p,R}_{{\rm Sq}}(B_R)}
\eeq
where we have used the fact $\#\{\nu\} \approx R^{1/2}$.
Therefore we have
\beq\label{eq:175}
\mathbf{S}_{\bf A}^{\sigma, \varepsilon}(\lambda,R)\lesssim  R^{\frac{1}{4}-\sigma},
\eeq
which will also serve as a starting point  of our induction.

 The proof of \eqref{eq:151} and  \eqref{eq:152} is reduced to
\beq
\mathbf{S}_{\bf A}^{\sigma(p), \varepsilon}(\lambda,\lambda^{1-\varepsilon/2})\lesssim_{\varepsilon, M, p }1,\quad p=3,4.
\eeq
Actually, since  ${\rm supp}_z\,a_\lambda(\cdot,\eta)$ is contained in the ball $B(0,\lambda)$, we may tile $B(0,\lambda)$ by a collection of balls of radius $\lambda^{1-\varepsilon/2}$ and use
 ${\mathbf S}_{\bf A}^{\sigma(p), \varepsilon}(\lambda,\lambda^{1-\varepsilon/2})\lesssim_{ \varepsilon, M, p} 1$ in each of the balls. By choosing appropriate $M$ depending $\varepsilon, L$ and  summing over all generated  balls,  we will obtain the desired result.

One key ingredient in the proof will be parabolic rescaling argument.
Let $\alpha\in \R$, we define the notion of $r-$ plate as follows
\[
\Gamma_\alpha^r:=\Bigl\{(\eta_1,\eta_2)\in{\rm supp}_\eta \,a:\bigl|\frac{\eta_1}{\eta_2}-\alpha\bigr|\leqslant r\Bigr\}.
\]
\begin{proposition}\label{proa}
	Let phase $\phi$ satisfy the condition $\mathbf{H}_{1}$,$\mathbf{H}_{2}$, amplitude $a$ satisfy $\mathbf{M}_{\bf A}$. Let  $1\leq \rho\leq R\leq \lambda$ and  $\mathscr{T}_\lambda$ be defined associated with  $\phi$ and  $a$. If $f$ is supported on a $\rho^{-1}-$plate with  $\rho$ is sufficiently large depending on $\phi$, then we have
	\beq\label{eq:2.4}
	\|\mathscr{T}_\lambda f\|_{L^p(B_R)}\lesssim_{\phi,a}  \mathbf{S}_{{\bf 1}}^{\sigma, \varepsilon}\Bigl(\frac{\lambda}{\rho^2},\frac{R}{\rho^2}\Bigr)\Bigl(\frac{R}{\rho^2}\Bigr)^{\sigma+\varepsilon}\|\mathscr{T}_\lambda f\|_{L^{p,R}_{\rm Sq}(B_R)}+\rho^{-2} R^4 \big(\frac{\lambda}{R}\big)^{-2M}\|f\|_{L^2(\R^2)}.
	\eeq
\end{proposition}

\begin{proof}
	The conditions  $\mathbf{H}_1, \mathbf {H}_2$
	imply  that there exists a special coordinate system,
	so that the phase function $\phi(z,\eta)$ can be written in a \emph{normalized} form.
	More precisely,
	according to Lemma 3.4 in \cite{Lee-JFA}, by appropriate affine transformation we may assume  $G(\mathbf{0},\mathbf{e}_2)=\mathbf{e}_2$ and $\partial_t\partial_{\eta_1}^2\phi(\mathbf{0},\mathbf{e}_2)=1$,
	then up to multiplying harmless factors to $\mathscr{T}_\lambda f$ and $f$,
	we can write $\phi$ in this coordinate as
	\beq\label{eq:18}
	\phi(x,t,\eta)=\langle x,\eta\rangle
	+\frac{t}{2}\,
	\eta_1^2/\eta_2
	+\eta_2\,\mathcal{E}(x,t,\eta_1/\eta_2)
	\eeq
	where $\eta=(\eta_1,\eta_2) $ and  $\mathcal{E}(x,t,s)$ obeys
	\beq\label{eq:14}
	\mathcal{E}(x,t,s)\leq C_{\phi}\bigl((|x|+|t|)^2|s|^2+(|x|+|t|)|s|^3\bigr).
	\eeq
	provided ${\rm supp}\; a$ is sufficiently small, otherwise, we may decompose the ${\rm supp} \; a$ into several small pieces such that \eqref{eq:18},\eqref{eq:14} hold.
	For further details in this direction, one may refer to \cite{Bo,Ho, GHI}. We perform to $a(z,\eta)$ the same transforms which convert $\phi$ into  its {\emph normalized} form in \eqref{eq:18}, and we denote it still by $a(z,\eta)$.
	Let the operator  $T_\lambda$  be defined  as follows
	\beq
	T_\lambda f(z,t):=\int e^{i\lambda\phi(z,\eta)}a(z,\eta)f(\eta) \dif \eta.
	\eeq
	As a result, by changing of variable: $z\to \lambda z$, we have
	\beq
	\|\mathscr{T}_\lambda f\|_{L^p(B_R)}\lesssim_{\phi} \lambda^{\frac3p} \|T_\lambda f\|_{L^p(B_{R/\lambda})}.
	\eeq

	Assume $f$ is supported in a plate $\Gamma_\gamma^{\rho^{-1}}$.
	By changing of variable: $\eta_1\rightarrow \eta_1+\gamma \eta_2$, we may rotate the central axis of $\Gamma_\gamma^{\rho^{-1}}$ to the $\mathbf{e}_2$ axis.  We make the associated change of variable in the physical space
	\begin{equation} \label{eq:220}
	\left\{\begin{aligned}
	x_1+\gamma t =x_1^{(1)};\\
	\gamma x_1+x_2+\frac{1}{2}\gamma^2t=x_2^{(1)};\\
	t= t^{(1)}.
	\end{aligned}\right.
	\end{equation}
	For convenience, we will use $x^{(1)}$ to denote $x^{(1)}:=(x_1^{(1)}, x_2^{(1)})$. Since  the transformation above is a diffeomorphism, we may use
	$\Phi^{(1)}(x^{(1)}, t^{(1)})$ to denote the inverse map of \eqref{eq:220}.
	In  the new variable system, the phase  $\phi$ is transformed to
	\beq \label{eq:121}
	\phi^{(1)}(x^{(1)}, t^{(1)}, \eta)=\langle x^{(1)},\eta\rangle+\frac{1}{2}t^{(1)} \eta_1^2/\eta_2+\eta_2\mathcal{E}_1( x^{(1)}, t^{(1)}, \eta_1/\eta_2+\gamma),
	\eeq
	where we denote $ \mathcal{E}_1(x^{(1)}, t^{(1)}, \eta_1/\eta_2+\gamma)=\mathcal{E}( \Phi^{(1)}(x^{(1)}, t^{(1)}), \eta_1/\eta_2+\gamma)$ the error term.
	
	Next make  Taylor's expansion of $\mathcal{E}_1(x^{(1)},t^{(1)},  \eta_1/\eta_2+\gamma )$ as follows
	\begin{align*}
	\mathcal{E}_1(x^{(1)}, t^{(1)} ,  \eta_1/\eta_2+\gamma)&=\mathcal{E}_{1}(x^{(1)},t^{(1)} , \gamma)+\partial_s\mathcal{E}_{1}(x^{(1)},t^{(1)} , \gamma)\frac{\eta_1}{\eta_2}\\
	&+\frac{1}{2}\partial_s^2\mathcal{E}_1(x^{(1)},t^{(1)},  \gamma )\big(\frac{\eta_1}{\eta_2}\big)^2\\
	&+\frac{1}{2}\int_0^1 \partial_s^3\mathcal{E}_{1}(x^{(1)},t^{(1)}, s\frac{\eta_1}{\eta_2}+\gamma)\big(\frac{\eta_1}{\eta_2}\big)^3(1-s)^2 \dif s.
	\end{align*}
	Making change of variables
	\begin{equation}\label{eq:195}
	\left\{\begin{aligned}
	x_1^{(1)}+\partial_{s}\mathcal{E}_1(x^{(1)}, t^{(1)},   \gamma)&=x_1^{(2)}\\
	\mathcal{E}_{1}(x^{(1)},t^{(1)} , \gamma)+x_2^{(1)}&=x_2^{(2)}\\
	\frac{1}{2}t^{(1)}+\frac{1}{2}\partial_{s}^2\mathcal{E}_{1}(x^{(1)}, t^{(1)}, \gamma )&= \frac{1}{2} t^{(2)}.
	\end{aligned}\right.
	\end{equation}
	As above, we use $x^{(2)}$ to denote $x^{2}:=(x_1^{(2)}, x_2^{(2)})$ and  $\Phi^{(2)}(x^{(2)}, t^{(2)})$ to denote the inverse map of \eqref{eq:195}.
	
	Accordingly, the phase function in \eqref{eq:121} is changed to
	\begin{align}\label{eq:19}
	\begin{split}
	\phi^{(2)}(x^{(2)}, t^{(2)}, \eta)&=\langle x^{(2)},\eta\rangle+\frac{1}{2} t^{(2)}\eta_1^2/\eta_2+\eta_2\mathcal{E}_{2}(x^{(2)}, t^{(2)},\eta_1/\eta_2).
	\end{split}
	\end{align}
	where
	\beq \label{eq:171}
	\mathcal{E}_2(x^{(2)}, t^{(2)}, \frac{\eta_1}{\eta_2})=\frac{1}{2}\int_0^1 \partial_s^3\mathcal{E}\Big(\Phi^{(1)}\circ\Phi^{(2)}(x^{(2)}, t^{(2)}), s\frac{\eta_1}{\eta_2}+\gamma\Big)\big(\frac{\eta_1}{\eta_2}\big)^3(1-s)^2 \dif s.
	\eeq
	Define $\Phi=\Phi^{(1)}\circ \Phi^{(2)}$, the amplitude $a$ is changed to
	\beq
	a_2(x^{(2)}, t^{(2)}, \eta)= a(\Phi(x^{(2)}, t^{(2)}), \eta_1+\gamma \eta_2,\eta_2 ).
	\eeq
	Let $\tilde {T}_\lambda $ be defined as  follows
	\begin{align*}
	\tilde {T}_{\lambda}f&=\int e^{i \lambda  \phi^{(2)}(z ,\eta)} a_2 (z,\eta)f(\eta)\dif \eta.
	\end{align*}
	Therefore,
	\beq
	\|T_\lambda f\|_{L^p(B_{R/\lambda})}
	\lesssim
	\|\tilde{T}_\lambda f\|_{L^p(\Phi^{-1}(B_{R/\lambda}))}.
	\eeq

	Let $\{R_{\Lambda}\}_{\Lambda}\subset \R^3 $ be a collection of pairwise disjoint rectangles of sidelength
$\rho^{-1} \frac{R}\lambda\times \rho^{-2} \frac{R}\lambda\times \frac{R}\lambda$ satisfying
	\beq
	 \Phi^{-1}(B_{R/\lambda}) \subset \bigcup_{\Lambda} R_\Lambda.
	\eeq
	By orthogonality of $R_\Lambda$, we have
	\beq
	\|\tilde{ T}_{\lambda}f\|_{L^p(\Phi^{-1}(B_{R/\lambda}))}^p \lesssim \sum_{\Lambda}\|\tilde{ T}_{\lambda}f\|_{L^p(R_\Lambda)}^p.
	\eeq
	Finally,  by scaling $\eta_1\rightarrow \rho^{-1} \eta_1$, the corresponding map for the amplitude  $$\eta  \rightarrow a_2(z, \rho^{-1}\eta_1,\eta_2)$$   which is supported  in
	\begin{align*}
	\mathscr{D}&=\Big\{(\eta_1,\eta_2)\in\R^2: |\eta_1/\eta_2|\leqslant 1/10, |\eta_2-1|\leqslant \varepsilon_0\Big\}.
	\end{align*}
	
		For each $R_\Lambda$, by changing   variables
$$(x,t)\rightarrow\bigl(z_\Lambda+(\lambda^{-1}\rho x_1, \lambda^{-1}x_2,  \lambda^{-1}\rho^2 t )\bigr),$$
with $z_{\Lambda}$ being the center of $R_{\Lambda}$.	we have
	\beq
	\|\tilde{ T}_\lambda f\|_{L^p(R_\Lambda)}
	\lesssim \lambda^{-\frac{3}{p}}\rho^{\frac3p}
	\|\mathscr{\tilde T}_{ \lambda \rho^{-2}}f\|_{L^p(B_{ \rho^{-2} R})},
	\eeq
	where
	\begin{align*}
	&\mathscr{\tilde T}_{\lambda} f(z)=\sum_\nu \tilde{\mathscr{T}}^\nu_\lambda f(z),\;
	\tilde{\mathscr{T}}^\nu_\lambda f(z)=\int e^{i \lambda {\tilde \phi}(\f{z}{\lambda} ,\eta)}\tilde{a}^\nu (z, \eta)f(\eta)\dif \eta,\\
	&\tilde{\phi}(z,\eta)= x_1\eta_1+x_2\eta_2+\frac{1}{2}t\eta_1^2/\eta_2+\rho^2\eta_2
\mathcal{E}_2(\Phi\bigl(z_\Lambda+(\rho^{-1} x_1, \rho^{-2}x_2,  t ), \rho^{-1}\eta_1/\eta_2\big),\\
	&\tilde{a}^\nu (z,\eta)=\tilde a(z,\eta)\chi_\nu(\rho^{-1}\eta_1+\gamma \eta_2,\eta_2),\,
	\tilde{a} (z,\eta)=a_2\bigl(z_\Lambda+(\rho^{-1} x_1, \rho^{-2}x_2, t ), \rho^{-1}\eta_1, \eta_2\big).
	\end{align*}
	
	If the datum $(\tilde \phi, \tilde a)$ is of type ${\bf 1}$, we may apply \eqref{eq:101} directly to obtain
	\begin{align}
	\nonumber
	\|\mathscr{\tilde T}_{\lambda\rho^{-2}}f\|_{L^p(B_{\rho^{-2}R})}\lesssim &  \mathcal{S}_1^{\sigma, \varepsilon}\Big(\frac{\lambda}{\rho^2},
\frac{R}{\rho^2}\Big) \Bigl(\frac{R}{\rho^2}\Bigr)^{\sigma+\varepsilon} \Bigl\|\Big(\sum_{\nu}|\mathscr{\tilde T}_{\lambda\rho^{-2}}^\nu f|^2\Big)^{\frac12}\Bigr\|_{L^p(w_{B_{R\rho^{-2}}})}\\
&+\rho^{-8-4M}R^4 \Bigl(\frac{\lambda}{R}\Bigr)^{-2M}\|f\|_{L^2(\R^2)}.
	\label{eq:196}\end{align}
	Reversing the above transforms  and summing over $\Lambda$'s , we obtain \eqref{eq:2.4}.

Now it remains to show the datum $(\tilde \phi, \tilde a)$ is of type ${\bf 1}$, note that \eqref{eq:19} and \eqref{eq:171}, by replacing the $\tilde \phi$ with $\tilde \phi-\rho^2\eta_2
\mathcal{E}_2\big(\Phi(z_\Lambda, \rho^{-1}\eta_1/\eta_2)\big)$, $\mathbf{ H}_{{\bf 1}})$,  $\mathbf{ D}_{{\bf 1}})$ can be readily verified by choosing $\rho$ sufficiently large depending on $\phi$. It is an issue that the hypothesis $\mathbf{ M}_{\bf 1}$ may break down, however, this can be resolved by decomposing ${\rm supp}_z  a$ into finitely  many small balls and translating to the origin, as a results, one may obtain a number of operators each defined associated with type $\mathbf{1}$ datum. It follows by using \eqref{eq:101} on each of ball, then summing over all the generated balls.
\end{proof}
\begin{remark}
	Throughout the proof of Proposition  \ref{proa}, one may reduce type $\mathbf{A}$ datum to type $\mathbf{1}$ datum.  Hence from now on, we shall always assume $(\phi, a)$ is of type $\mathbf{1}$ without specified clarification.
\end{remark}
\section{Linear v.s.  Multilinear square function estimates }\label{sect:3}
In this section, by modification of the argument in \cite{LeVa12}, we establish the relation between linear and multilinear square function estimates.  In order to describe the  setup, it requires the notion of  transversality.
\begin{definition}
	Let $\mathbf{T}_\lambda=(\mathscr{T}^1_\lambda, \mathscr{T}^2_\lambda, \mathscr{T}^3_\lambda)$ be $3-$tuple of oscillatory integral operators satisfying the
cinematic  curvature  conditions, where $\mathscr{T}_\lambda^j$ has associated phase $\phi^\lambda_j$, amplitude $a_\lambda^j$ and generalised Gauss map  $G_j$ for $1\leqslant j\leqslant 3$. Then $\mathbf{T}_\lambda$ is said to be $\vartriangle$ transverse for some $0< \vartriangle\leqslant 1$ if
	\beq \label{eq:223}
	\Bigl|\bigwedge_{j=1}^3G_j(z,\eta_j)\Bigr|\geqslant \vartriangle\quad \text{for all}\quad (z,\eta_j)\in {\rm supp}\; a_j\, ,\quad \;1\leqslant j\leqslant 3.
	\eeq
\end{definition}

Let $2\leq p< \infty, 1\leq R\leq \lambda^{1-\varepsilon/2}$, ${\bf T}_\lambda=(\mathscr{T}_\lambda^1,\mathscr{T}_\lambda^2,\mathscr{T}_\lambda^3)$ are $\vartriangle$ transverse, we use  $\mathbf{MS}_{\bf 1}^{\sigma, \varepsilon}(\lambda,R)$ to denote the sharp constant such that the following trilinear square function estimate
\beq \label{eq:100}
\Big\|\prod_{j=1}^3 |\mathscr{T}_{\lambda}^jf|^{\frac{1}{3}}\Big\|_{L^{p}(B_R)}\leqslant \mathbf{ MS}_{\bf 1}^{\sigma, \varepsilon}(\lambda,R) R^{\sigma+\varepsilon} \prod_{j=1}^3 \|  \mathscr{T}_{\lambda}^{j}f\|_{L^{p,R}_{\rm  Sq}(B_R)}^{\frac{1}{3}}+R^2 \Big(\frac{\lambda}{R}\Big)^{-3M}\|f\|_{L^2},
\eeq
holds for all type {\bf 1} data $(\phi_j, a_j), j=1,2,3$.

In the seminal paper \cite{BoGu11}, Bourgain and Guth initiated the multilinear approach towards Fourier restriction conjectures and oscillatory integral estimates based on a fundamental result of \cite{BCT06}.
This strategy was later adapted to the cone multiplier problem by Lee and Vargas in \cite{LeVa12}.  Motivated by these previous works, we adopt the similar idea to deduce
linear square function estimates from the multilinear one. The following lemma can be established by using Bourgain-Guth's method and we omit the details (see also pp930-934 in \cite{LeVa12} for details).
\begin{lemma} \label{le1}
	Let $1\ll K_1\ll K_2\ll R$ and  $\{\theta_\mu\}_\mu, \{\theta_\tau\}_\tau$ be two families of plates of aperture $K^{-1}_1$ and $K_2^{-1}$, respectively.  We make decomposition of $\mathscr{T}_{\lambda}f$ as follows
	\beq
	\mathscr{T}_{\lambda}f=\sum_{\mu:\theta_\mu \in \{\theta_\mu \}}\mathscr{T}_{\lambda}^\mu f =\sum_{\tau:\theta_\tau \in \{\theta_\tau \}}\mathscr{T}_{\lambda}^\tau f,
	\eeq
	where \begin{align}
	\mathscr{T}_{\lambda}^\mu f(z):=\int e^{i\phi^\lambda(x,t, \eta)} a_\lambda^{\mu}(x,t,\eta)f(\eta) \dif \eta, \;a_\lambda^\mu(z,\eta)=\sum\limits_{\nu:\,\theta_\nu \subset \theta_\mu} a_\lambda^\nu(z,\eta),\\
	\mathscr{T}_{\lambda}^\tau f(z):=\int e^{i\phi^\lambda(x,t, \eta)} a_\lambda^{\tau}(x,t,\eta)f(\eta) \dif \eta, \;a_\lambda^\tau(z,\eta)=\sum\limits_{\nu:\, \theta_\nu\subset\theta_\tau } a_\lambda^\nu(z,\eta).
	\end{align}
	Then for any $ z \in {\rm supp} \; a_\lambda$, we have
	\begin{equation}
	\label{eq:181}
	\begin{split}
	|\mathscr{T}_{\lambda} f(z)|\leqslant C \max_{\mu} |\mathscr{T}_{\lambda}^\mu f(z)|&+C K_1\max_{\tau}|\mathscr{T}_{\lambda}^\tau f(z)|\\
	&+K_2^{50}\max_{(\tau_1,\tau_2,\tau_3)\in{\rm Tr}(\tau)}\prod_{j=1}^3|\mathscr{T}_{\lambda}^{\tau_j}f(z)|^{1/3},
	\end{split}
	\end{equation}
	where
	${\rm Tr}({\bf \tau}):=\{(\tau_1,\tau_2,\tau_3): \theta_{\tau_j} \in \{\theta_\tau\}, j=1,2,3,\; \text{and} \;\mathsf{ Ang}(\theta_{\tau_j},\theta_{\tau_{j'}})\geq K^{-1}_2, j\neq j'\}$.
\end{lemma}
Based on the Lemma \ref{le1}, we relate the trilinear estimate \eqref{eq:100} to \eqref{eq:101}.
\begin{proposition}\label{prob}
	Suppose {\bf MS}$_{\bf 1}^{\sigma, \varepsilon}(\lambda,R)\lesssim_\varepsilon 1$ for all $(\lambda,R)$ satisfying $1\leq R\leq \lambda^{1-\varepsilon/2}$. 
	Then  $${\bf S}_{\bf 1}^{\sigma, 2\varepsilon}(\lambda,R)\lesssim_\varepsilon  1,$$
	whenever $(\lambda,R)$ fulfills the condition $1\leqslant R\leqslant \lambda^{1-\frac{\varepsilon}{2}}$.
\end{proposition}

\begin{proof}
	If $1\leq R \leq 100$, then  by \eqref{eq:175}, we have ${\bf S}_{\bf 1}^{\sigma,2\varepsilon}(\lambda,R)\leq C$. We proceed by induction argument. Suppose
	\beq \label{eq:228}
	{\bf S}_{\bf 1}^{\sigma, 2\varepsilon}(\lambda',R')\lesssim_\varepsilon 1 \; \text{for all}\; R'\leq R/2,\;   R'\leq (\lambda')^{1-\varepsilon/2},
	\eeq
	it suffices to show
	\beq
	{\bf S}_{\bf 1}^{\sigma, 2\varepsilon}(\lambda,R)\lesssim_\varepsilon 1 \; \text{for all}\;  R\leq \lambda^{1-\varepsilon/2}.
	\eeq
	To this end, using  Lemma \ref{le1}, we have
	\begin{align}
	\|\mathscr{T}_{\lambda} f\|_{L^p (B_R)}^p \leq  &C \sum_{\mu:\theta_\mu \in \{\theta_\mu \}}\|\mathscr{T}_{\lambda}^\mu f\|_{L^p(B_R)}^p+ C K_1^p \sum_{\tau: \theta_\tau  \in \{\theta_\tau \}}\|\mathscr{T}_{\lambda}^\tau f\|_{L^p(B_R)}^p \label{eq:109}\\ &+K_2^{50p}\sum_{(\tau_1,\tau_2,\tau_3)\in{\rm Tr}(\tau)}\|\prod_{j=1}^3|\mathscr{T}_{\lambda}^{\tau_j}f|^{1/3}\|_{L^p(B_R)}^p.\label{eq:112}
	\end{align}
	By Proposition \ref{proa}, we have
	\begin{align}
	\|\mathscr{T}_{\lambda}^\mu f\|_{L^p (B_R )}\leq &\,C\, {\rm S}_{{\bf 1}}^{\sigma, 2\varepsilon}
\Big(\frac{\lambda}{K_1^2}, \frac{R}{K_1^2}\Big)\Big(\frac{R}{K_1^2}\Big)^{\sigma+2\varepsilon}  \Big\|\Big(\sum_{\nu: \theta_\nu \subset \theta_\mu}|\mathscr{T}_{\lambda}^{\nu}f|^2\Big)^{\frac12}\Big\|_{L^p (w_{B_R})}\nonumber\\
	&+ K_1^{-2}R^4 (\frac{\lambda}{R})^{-2M}\|f\|_{L^2}.
	\end{align}
	Raising to the $p$th power and summing over all $\mu$'s, we get
	\begin{align}
	\Big(\sum_{\mu:\theta_{\mu}\in\{\theta_{\mu}\}}\Big\|\mathscr{T}_{\lambda}^\mu f\Big\|_{L^p (B_R)}^p\Big)^{\frac1p}
	\leq& C  S_{{\bf 1}}^{\sigma, 2\varepsilon}\Bigl(\frac{\lambda}{K_1^2},\frac{R}{K_1^2}\Bigr)\Big(\frac{R}{K_1^2}\Big)^{\sigma+2\varepsilon}
\Big\|\Big(\sum_{\nu }|\mathscr{T}_{\lambda}^{\nu}f|^2\Big)^{\frac12}\Big\|_{L^p (w_{B_{R}})}\nonumber
	\\
	& + K_1^{-1}R^4 \Big(\frac{\lambda}{R}\Big)^{-2M}\|f\|_{L^2}\label{eq:226}.
	\end{align}
	Similarly, we obtain the  corresponding estimate for the second term in \eqref{eq:109}.
	\begin{align}
	\Big(\sum_{\tau:\theta_{\tau}\in\{\theta_{\tau}\}} \Big\|\mathscr{T}_{\lambda}^\tau f\Big\|_{L^p (B_R )}^p\Big)^{\frac1p}\leq &CK_1{\rm S}_{{\bf 1}}^{\sigma, 2\varepsilon}\Bigl(\frac{\lambda}{K_2^2},\frac{R}{K_2^2}\Bigr)\Bigl(\frac{R}{K_2^2}\Bigr)^{\sigma+2\varepsilon}  \Big\|\Big(\sum_{\nu }|\mathscr{T}_{\lambda}^{\nu}f|^2\Big)^{\frac12}\Big\|_{L^p (w_{B_\lambda})}\nonumber\\
	&+K_1K_2^{-1}R^4 \Big(\frac{\lambda}{R}\Big)^{-2M}\|f\|_{L^2}.\label{eq:227}
	\end{align}

	It remains to estimate \eqref{eq:112}.
	In order to use the trilinear square function estimate \eqref{eq:100} with ${\rm  MS}_{\bf 1}^{\sigma,\varepsilon}(\lambda,R)\lesssim 1$, it requires the verification of the transversality condition.
	
	Let $(\tau_1^0,\tau_2^0,\tau_3^0)\in {\rm Tr}(\tau)$ be such that
	\[
	\Bigl\|\prod_{j=1}^3|\mathscr{T}_{\lambda}^{\tau^0_j}f|^{\frac13}\Bigr\|_{L^p(B_R)}^p
	=\max_{(\tau_1,\tau_2,\tau_3)\in{\rm Tr}(\tau)}\Bigl\|\prod_{j=1}^3|\mathscr{T}_{\lambda}^{\tau_j}f|^{\frac13}\Bigr\|_{L^p(B_R)}^p,
	\]
	Then \eqref{eq:112}
	is bounded by
	\beq\label{eq:ooo}
	K^{O(1)}_2\Bigl\|\prod_{j=1}^3|\mathscr{T}_{\lambda}^{\tau^0_j}f|^{\frac13}\Bigr\|_{L^p(B_R)}^p.
	\eeq
	since the cardinality of the set
	${\rm Tr}({\bf \tau})$
	is at most $K^{3}_2$.
	By change of variable $z\to \lambda z$, \eqref{eq:ooo} becomes
	\beq \label{eq:200}
	\lambda^{3}K^{O(1)}_2\bigl\|\prod_{j=1}^3|T_{\lambda}^{\tau^0_j}f|^{1/3}\bigr\|_{L^p(B_{R/\lambda})}^p
	\eeq
	where for $j=1,2,3$,
	$$T_{\lambda}^{\tau^0_j}f(x,t)=\int e^{i\lambda\phi(x,t,\eta)}a^{\tau_j^0}(x,t,\eta)f(\eta)\dif \eta,\;
	a^{\tau_j^0}(x,t,\eta)=a_\lambda^{\tau_j^0}(\lambda (x,t),\eta).$$

	Let $\eta^j=(\eta_1^j,\eta_2^j)$ be the center of $\theta_{\tau^0_j}$ with $j=1,2,3$ and set
	$\alpha_j=\eta_1^j/\eta^j_2$.
	Then, for every  $\eta \in \theta_{\tau^0_j}$,  we have $$\left|\frac{\eta_1}{\eta_2}-\alpha_j \right|\leq K_2^{-1},\;j=1,2,3.$$  Due to the  angular separation condition involved in ${\rm Tr}(\tau)$, we may assume $\alpha_1>\alpha_2>\alpha_3$ after a rearrangement if necessary. Denote
	\beq \label{eq:111}
	\alpha:=\min\{\alpha_1-\alpha_2, \alpha_2-\alpha_3\}.
	\eeq
	Then we have  $\alpha\geq K^{-1}_2$.
	
	In order to verify the transversality condition, we need to make another transform with respect to  $\eta_1$ to ensure the angular separation $\approx 1$ independent of $\varepsilon_0$.
	To this end,  one may modify the argument in the proof of Proposition \ref{proa}.

	We make a change of variable  $\eta_1\rightarrow \eta_1+\alpha_2 \eta_2$ in each oscillatory integral $T_{\lambda}^{\tau^0_j}f$ with respect to the frequency space for all $j=1,2,3$.
	In the mean time, we make an affine transform in \eqref{eq:200} with respect to the $(x,t)-$variables
	\begin{equation}\label{eq:224}
	\left\{\begin{aligned}
	x_1+\alpha_2 t =x_1^{(1)}\\
	\alpha_2 x_1+x_2+\frac{1}{2}\alpha_2^2t=x_2^{(1)}\\
	t= t^{(1)}
	\end{aligned}\right.
	\end{equation}

	We denote $x^{(1)}:=(x_1^{(1)}, x_2^{(1)})$ and use
	$\Phi^{(1)}(x^{(1)}, t^{(1)})$ to denote the inverse map of \eqref{eq:224}.
	In the new variable system, the phase function $\phi$ is transformed to
	\beq \label{eq:121-add}
	\phi^{(1)}(x^{(1)}, t^{(1)}, \eta)=\langle x^{(1)},\eta\rangle+\frac{1}{2}t^{(1)} \eta_1^2/\eta_2+\eta_2\mathcal{E}_1( x^{(1)}, t^{(1)}, \eta_1/\eta_2+\alpha_2),
	\eeq
	where $$ \mathcal{E}_1(x^{(1)}, t^{(1)}, \eta_1/\eta_2+\alpha_2)=\mathcal{E}( \Phi^{(1)}(x^{(1)}, t^{(1)}), \eta_1/\eta_2+\alpha_2).$$
	Taylor expanding  $\mathcal{E}_1(x^{(1)},t^{(1)},  \eta_1/\eta_2+\alpha_2 )$ as follows
	\begin{align*}
	\mathcal{E}_1(x^{(1)}, t^{(1)} ,  \eta_1/\eta_2+\alpha_2)=&\mathcal{E}_{1}(x^{(1)},t^{(1)} , \alpha_2)+\partial_s\mathcal{E}_{1}(x^{(1)},t^{(1)} , \alpha_2)\frac{\eta_1}{\eta_2}\\
	&+\frac{1}{2}\partial_s^2\mathcal{E}_1\Bigl(x^{(1)},t^{(1)},  \alpha_2 \Bigr)\big(\frac{\eta_1}{\eta_2}\big)^2\\
	&+\frac{1}{2}\int_0^1 \partial_s^3\mathcal{E}_{1}
	\Bigl(x^{(1)},t^{(1)}, s\frac{\eta_1}{\eta_2}+\alpha_2\Bigr)\big(\frac{\eta_1}{\eta_2}\big)^3(1-s)^2 \dif s.
	\end{align*}
	
	Making an additional change of variables
	\begin{equation}\label{eq:211}
	\left\{\begin{aligned}
	x_1^{(1)}+\partial_{s}\mathcal{E}_1(x^{(1)}, t^{(1)},   \alpha_2)&=x_1^{(2)},\\
	\mathcal{E}_{1}(x^{(1)},t^{(1)} , \alpha_2)+x_2^{(1)}&=x_2^{(2)},\\
	\frac{1}{2}t^{(1)}+\frac{1}{2}\partial_{s}^2\mathcal{E}_{1}(x^{(1)}, t^{(1)}, \alpha_2 )&= \frac{1}{2} t^{(2)},
	\end{aligned} \right.
	\end{equation}
	the phase function in \eqref{eq:121-add} is changed to
	\begin{align}\label{eq:19add}
	\begin{split}
	\phi^{(2)}(x^{(2)}, t^{(2)}, \eta)&=\langle x^{(2)},\eta\rangle+\frac{1}{2} t^{(2)}\eta_1^2/\eta_2+\eta_2\mathcal{E}_{2}(x^{(2)}, t^{(2)},\eta_1/\eta_2),
	\end{split}
	\end{align}
	where
	\beq \label{eq:171add}
	\mathcal{E}_2\Bigl(x^{(2)}, t^{(2)}, \frac{\eta_1}{\eta_2}\Bigr)=\frac{1}{2}\int_0^1 \partial_s^3 \mathcal{E}\Bigl(\Phi^{(1)}\circ\Phi^{(2)}(x^{(2)}, t^{(2)}), s\frac{\eta_1}{\eta_2}+\alpha_2\Bigr)\big(\frac{\eta_1}{\eta_2}\big)^3(1-s)^2 \dif s,
	\eeq
	with $x^{(2)}:=(x_1^{(2)}, x_2^{(2)})$
	and  $\Phi^{(2)}(x^{(2)}, t^{(2)})$  denoting the inverse map of \eqref{eq:211}.

	Setting $\Phi=\Phi^{(1)}\circ \Phi^{(2)}$, the amplitude $a^{\tau_j^0}$ is changed by the above affine transforms  to
	\beq
	a^{\tau_j^0}(\Phi(x^{(2)}, t^{(2)}), \eta_1+\alpha_2 \eta_2,\eta_2 ):= a_2^{\tau_j^0}(x^{(2)}, t^{(2)}, \eta).
	\eeq
	We introduce an operator $\widetilde{ T}_\lambda^{\tau_j^0}$
	\begin{align*}
	\widetilde {T}_{\lambda}^{\tau_j^0}f(x^{(2)},t^{(2)})&=\int e^{i \lambda  \phi^{(2)}(x^{(2)},t^{(2)} ,\eta)} a_2^{\tau_j^0} (x^{(2)},t^{(2)},\eta)f(\eta)\dif \eta.
	\end{align*}
	Therefore,
	\beq
	\Bigl\|\prod_{j=1}^3|T_{\lambda}^{\tau^0_j}f|^{\frac13}\Bigr\|_{L^p(B_{R/\lambda})}^p
	\lesssim
	\Bigl \|\prod_{j=1}^3 |\widetilde{T}_\lambda^{\tau_j^0}  f|^{\frac13}\Bigr\|_{L^p(\Phi^{-1}(B_{R/\lambda}))}.
	\eeq

	Let $\{R_{\Lambda}\}_{\Lambda}\subset \R^3 $ be a pairwise disjoint collection of rectangles of sidelength $\alpha\frac{R}{\lambda} \times \alpha^2\frac{R}{\lambda}\times\frac{R}{\lambda}$, where the first two components correspond to the $x$-direction, such that
	\beq
	\Phi^{-1}(B_{R/\lambda}) \subset \bigcup_{\Lambda} R_\Lambda,
	\eeq
	By almost orthogonality, we have
	\beq
	\Bigl\|\prod_{j=1}^3|\widetilde{ T}_{\lambda}^{\tau_j^0}f|^{\frac{1}{3}}\Bigr\|_{L^p(\Phi^{-1}(B_{R/\lambda}))}^p \leq
\sum_{\Lambda}\Bigl\|\prod_{j=1}^3|\tilde{ T}_{\lambda}^{\tau_j^0} f|^{\frac{1}{3}}\Bigr\|_{L^p(R_\Lambda)}^p.
	\eeq
	
	Finally, under scaling $\eta_{1}\rightarrow \alpha \eta_{1}$,  the corresponding amplitude function becomes
	$$(\eta_1,\eta_2)  \rightarrow a_2^{\tau_j^0}\bigl(x^{(2)},t^{(2)}, (\alpha\eta_1,\eta_2)\bigr)$$
	which is supported respectively in $\mathscr{D}_j$ with
	\begin{align*}
	\mathscr{D}_1&=\{(\eta_1,\eta_2)\in\R^2: |\eta_1/\eta_2-\beta_1|\leqslant 1/10, |\eta_2-1|\leqslant \varepsilon_0\},\\
	\mathscr{D}_2&=\{(\eta_1,\eta_2)\in\R^2: |\eta_1/\eta_2|\leqslant 1/10, |\eta_2-1|\leqslant \varepsilon_0\},\\
	\mathscr{D}_3&=\{(\eta_1,\eta_2)\in\R^2: |\eta_1/\eta_2-\beta_3|\leqslant 1/10, |\eta_2-1|\leqslant \varepsilon_0\},
	\end{align*}
	where  $\beta_k=(\alpha_k-\alpha_2)\alpha^{-1}$ for $k=1,3$.
	From \eqref{eq:111},  it is easy to see $|\beta_k|\geq 1/2$ for $k=1,3$.
	\vskip0.2cm

	By changing   variables $(x^{(2)},t^{(2)})\rightarrow z_\Lambda+\bigl(\lambda^{-1}\alpha^{-1} x_1, \lambda^{-1}x_2,  \lambda^{-1}\alpha^{-2} t\bigr) $ with $z_\Lambda$ being the center of $R_\Lambda$,
	we have
	\beq
	\Bigl\|\prod_{j=1}^3|\tilde{ T}_{\lambda}^{\tau_j^0} f|^{\frac{1}{3}}\Bigr\|_{L^p(R_\Lambda)}\leqslant C\lambda^{-\frac3p}\alpha^{-\frac3p}
	\Big\|\prod_{j=1}^3|\mathscr{\tilde T}^{\tau_j^0}_{ \lambda \alpha^2}f|^{\frac{1}{3}}\Big\|_{L^p(B_{ \alpha^2 R})},
	\eeq
	where $B_{\alpha^2 R}$ is a ball of radius $\alpha^2R$ centered at the origin and
	\begin{align*}
&\mathscr{\tilde T}_{\lambda}^{\tau_j^0} f(z)=\sum_{\nu_j} \tilde{\mathscr{T}}^{\nu_j}_\lambda f(z),\;
	\tilde{\mathscr{T}}^{\nu_j}_\lambda f(z)=\int e^{i \lambda {\tilde \phi}(\f{z}{\lambda} ,\eta)}\tilde{a}^{\nu_j }(z, \eta)f(\eta)\dif \eta,\\
	&\tilde{\phi}(x,t,\eta)= x_1\eta_1+x_2\eta_2+\frac{1}{2}t\eta_1^2/\eta_2+\alpha^{-2}\eta_2\mathcal{E}_2(\Phi(z_\Lambda+(\alpha x_1, \alpha^2x_2,   t), \alpha\eta_1/\eta_2),\\
	&\tilde{a}^{\tau_j^0} (x,t,\eta)=a_2^{\tau_j^0}\bigl(z_\Lambda+(\alpha x_1, \alpha^2x_2,  t), (\alpha\eta_1,\eta_2)\bigr).
	\end{align*}
\vskip0.15cm
	\noindent  We have by direct calculation for $\varepsilon_0$ small enough
	\begin{align*}
	\begin{split}
	\partial_{x} \tilde \phi(x,t,\eta)&=\eta+O(\varepsilon_0|\eta|^2)+O(\varepsilon_0|\eta|^3),\\
	\partial_{t} \tilde \phi(x,t,\eta)&=\frac{1}{2}\eta_1^2/\eta_2+O(\varepsilon_0|\eta|^2)+O(\varepsilon_0|\eta|^3).
	\end{split}
	\end{align*}
	Therefore, it is easy to obtain the formula for  $G_0^j, j=1,2,3$
	\beq
	G^j_0(x,t,\eta)=\Bigl(-\frac{\eta_1}{\eta_2}, \frac{1}{2}\frac{\eta_1^2}{\eta_2^2}, 1\Bigr)+O(\varepsilon_0),\; \eta\in\mathscr{D}_j.
	\eeq
	From this, we obtain
	\beq
	|G^1_0\wedge G^2_0  \wedge G^{3}_0|\geqslant \frac{1}{8}|\beta_1||\beta_3||\beta_1-\beta_3|+O(\varepsilon_0)\geq \frac{1}{32}|\beta_1-\beta_3|^2 +O(\varepsilon_0),
	\eeq
	for all $(x,t,\eta)\in B(0,\varepsilon_0)\times \mathscr{D}_j \; j=1,2,3.$
	Since  $G^j:= \frac{G^j_0(z,\eta)}{|G^j_0(z,\eta)|}$, we have
	\beq
	|G^1\wedge G^2  \wedge G^3|\approx |\beta_1-\beta_3|^{-2}|G^1_0\wedge G^2_0  \wedge G^{3}_0| \geq c>0.
	\eeq  	
	then the transversality condition is guaranteed.  Following the approach in the proof of Proposition \ref{proa}, one may verify that the corresponding phase  and amplitude  $(\tilde \phi, \tilde a)$ is of type {\bf 1} datum.
	
	By our assumption, we have ${\rm MS}_{{\bf 1}}^{\sigma, \varepsilon}(\lambda\alpha^2, R\alpha^2)\lesssim_\varepsilon 1$.
	This yields
	\begin{align*}
	\Big\|\prod_{j=1}^3|\mathscr{\tilde T}^{\tau_j^0}_{ \lambda \alpha^2}f|^{\frac{1}{3}}\Big\|_{L^p(B_{ \alpha^2 R})}
	\leq& C_\varepsilon K_2^{O(1)}R^{\sigma+\varepsilon} \prod_{j=1}^3 \Big\| \Big(\sum_{\nu_j} |\tilde{\mathscr{T}}_{\lambda\alpha^2}^{\nu_j}f|^2\Big)^{\frac12}\Big\|_{L^{p}(w_{B_{\alpha^2 R}})}^{\f{1}{3}}\\
	&+K^{O(1)}_2 R^2 \Big(\frac{\lambda}{R}\Big)^{-3M}\|f\|_{L^2(\R^2)}.
   \end{align*}
		Undoing the transform and summing over $\Lambda$, we arrive at
	
	\begin{align}\nonumber
		\Bigl\|\prod_{j=1}^3|\mathscr{T}_{\lambda}^{\tau^0_j}f|^{\frac13}\Bigr\|_{L^p(B_R)}^p\leq&
 C_\varepsilon K_2^{O(1)}R^{(\sigma+\varepsilon)p}  \Big\| \Big(\sum_{\nu} |\mathscr{T}_{\lambda}^{\nu}f|^2\Big)^{\frac12}\Big\|_{L^{p}(w_{B_{ R}})}^p\\
		& +K^{O(1)}_2 R^{2p} \Big(\frac{\lambda}{R}\Big)^{-3Mp}\|f\|_{L^2(\R^2)}^p.
	\label{eq:225}\end{align}
		Collecting  \eqref{eq:109}\eqref{eq:112}\eqref{eq:226} \eqref{eq:227} \eqref{eq:225}, we obtain
	\begin{align*}
	\|\mathscr{T}_{\lambda} f\|_{L^p (B_R)}\leq&  \Bigl (C ({\rm S}_{{\bf 1}}^{\sigma, 2\varepsilon}(\lambda/K_1^2,R/ K_1^2)(1/K_1^2)^{\sigma+2\varepsilon}+C_\varepsilon K_2^{O(1)}R^{-\varepsilon}\\
	&+ C K_1^{O(1)}{\rm S}_{{\bf 1}}^{\sigma, 2\varepsilon}(\lambda/ K_2^2,R/ K_2^2)(1/K_2^2)^{\sigma+2\varepsilon} \Bigr)R^{\sigma+2\varepsilon}
\Bigl\| \Big(\sum_{\nu} |\mathscr{T}_{\lambda}^{\nu}f|^2\Big)^{\frac12}\Bigr\|_{L^{p}(w_{B_R})}\\
	&+(K_1^{-1}+K_1K_2^{-1} + K_2^{O(1)} R^{-M}R^{-2})R^4\Big(\frac{\lambda}{R}\Big)^{-2M}\|f\|_{L^2}.
	\end{align*}
	Choose $1\ll K_1$ such that $R/K_1^2\leq R/2$. Under the assumption  $R\leq \lambda^{1-\varepsilon/2}$, we have $$R/K_1^2 \leq (\lambda/K_1^2)^{1-\varepsilon/2}.$$
	By assumption \eqref{eq:228},  ${\rm S}_{{\bf 1}}^{\sigma, 2\varepsilon}(\lambda/K_1^2,R/ K_1^2)\leq C_\varepsilon $.  In view of $ 1\ll K_1\ll K_2\ll R$, we have
	\begin{gather}
	C \mathbf{S}_{{\bf 1}}^{\sigma, 2\varepsilon}\Bigl(\frac{\lambda}{K_1^2},\frac{R}{K_1^2}\Bigr)\Bigl(\frac{1}{K_1^2}\Bigr)^{\sigma+2\varepsilon}+ C K_1^{O(1)}\mathbf{S}_{{\bf 1}}^{\sigma, 2\varepsilon}\Bigl(\frac{\lambda}{K_2^2},\frac{R}{K_2^2}\Bigr)\Bigl(\frac{1}{K_2^2}\Bigr)^{\sigma+2\varepsilon}+C_\varepsilon K_2^{O(1)}R^{-\varepsilon} \leq C_\varepsilon,\notag\\
	K_1^{-1}+K_1K_2^{-1} + K_2^{O(1)} R^{-M}R^{-2}\leq 1.\notag
	\end{gather}
	Finally, we obtain
	\begin{align*}
	\|\mathscr{T}_{\lambda} f\|_{L^p (B_R)}\leqslant&  C_\varepsilon\Bigl\| \Big(\sum_{\nu} |\mathscr{T}_{\lambda}^{\nu}f|^2\Big)^{\frac12}\Bigr\|_{L^{p}(w_{B_R})}+R^4\Big(\frac{\lambda}{R}\Big)^{-2M}\|f\|_{L^2},
	\end{align*}
	which completes the proof.
\end{proof}

\section{$L^3-$variable coefficient square function estimate via multilinear approach}\label{sect:4}
In this section, we prove  for $1\leqslant R\leqslant \lambda^{1-\varepsilon/2}$
\beq \label{eq:110}
\Big\|\prod_{j=1}^3 |\mathscr{T}_{\lambda}^j  f|^{\frac{1}{3}}\Big\|_{L^{3}(B_R)}\leqslant C_\varepsilon R^\varepsilon  \prod_{j=1}^3 \Big\| \Big(\sum_{\nu_j} |\mathscr{T}_{\lambda}^{\nu_j}f|^2\Big)^{1/2}\Big\|_{L^{3}(w_{B_{R}})}+\lambda^{-4M}\|f\|_{L^2(\R^2)},
\eeq
which  implies ${\rm MS}_{\bf 1}^{0,\varepsilon}(\lambda, R)\lesssim_\varepsilon 1$.

The main ingredient of the proof is the Bennett-Carbery-Tao's  multilinear oscillatory integral estimates below.

\begin{theorem}[\cite{BCT06}]\label{theo2}
	Under the transversality condition  described as above, then for each $\varepsilon>0$ the datum $(\phi, a)$ is of type ${\bf A}$, then  there is a constant $C>0$, depending only on $\varepsilon,\vartriangle, A_2 $, for which
	\beq \label{eq:6}
	\Big\|\prod_{j=1}^3 |\mathscr{T}_{\lambda}^j f_j|^{\frac13}\Big\|_{L^3(B_R)}\lesssim_{\varepsilon, \vartriangle, A_2}  R^{\varepsilon}\prod_{j=1}^3 \|f_j\|_{L^2(\R^{2})}^{\frac{1}{3}}.
	\eeq
\end{theorem}
Before the proof of the main theorem, we take the constant case $e^{it\sqrt{-\Delta}}$ to give an intuition behind the proof.  Actually \eqref{eq:6} can be upgraded to the following  the refined version
\beq
\Big\|\prod_{j=1}^3 |e^{it\sqrt{-\Delta}} f_j|^{\frac13}\Big\|_{L^3(B_R)}\lesssim_\varepsilon   R^{-1/2+\varepsilon}\prod_{j=1}^3 \|e^{it\sqrt{-\Delta}}f_j\|_{L^2(w_{B_R})}^{\frac{1}{3}},
\eeq
where $\Omega_j:={\rm supp}\hat{f_j}$ are separated in the angular direction. Then  we have by orthogonality
\beq\label{eq:q1}
\|e^{it\sqrt{-\Delta}}f_j\|_{L^2(w_{B_R})}\leq C \Big(\sum_{\nu_j\cap \Omega_j \neq \varnothing}\|e^{it\sqrt{-\Delta}}f_{\nu_j}\|_{L^2(w_{B_R})}^2\Big)^{1/2},
\eeq
where $\hat{f}_{\nu_j}=\hat{f}\chi_{\nu_j}$.
 By   H\"older's inequality, we obtain
 \beqq
 \Big\|\prod_{j=1}^3 |e^{it\sqrt{-\Delta}} f_j|^{\frac13}\Big\|_{L^3(B_R)}\lesssim_\varepsilon   R^{\varepsilon}\prod_{j=1}^3 \Big\|(\sum_{\nu_j\cap \Omega_j\neq \varnothing}|e^{it\sqrt{-\Delta}}f_{\nu_j}|^2)^{\f{1}{2}}\Big\|_{L^3(w_{B_R})}^{\frac{1}{3}}.
 \eeqq

  For the variable coefficient setting,  we can't carry out the above procedure verbatim since there are some obstacles appearing.  The orthogonality property, among other things, will be obscure.  Actually, in the constant case, for different $j$, the orthogonality shared among $\nu_j$ is  obvious in the frequency space. Things are different for our setting since we can't identify the precise location in the frequency space for different $\nu_j$.

  To overcome this obstacle, we will further decompose the plate in the radial direction to obtain equally-spaced boxes so that the locally constant property can be exploited.

\begin{definition}[locally constant property]
	Assume $d\geqslant 1$, given a function $F: \R^d \rightarrow [0,\infty)$, we say $F$ satisfies the
	\emph{ locally constant  property} at scale of $\rho$ if $F(x)\approx F(y)$ whenever $|x-y|\leqslant C_0 \rho$. Here the implicit constant in ``\,$\approx$" could depend on the structure constant $C_0$.
\end{definition}
 Let us define  the extension operator $E$ to be
\beq
Ef(x,t):=\int_{\R^2} e^{i(x\eta+t h(\eta))} a_2(\eta) f(\eta) \dif \eta,
\eeq
where $h(\eta)$ is a smooth away from origin and homogeneous of degree $1$ with
\begin{equation}
{\rm rank}\;\partial^2_{\eta\eta} h=1,\quad \text{for all } \;\; \eta\in {\rm supp}\;a_2.
\end{equation}
Assume $r\geqslant 1$ and $f$ is supported in a  $r^{-1}$ neighborhood of $\eta_0 \in {\rm supp} \;a_2$, then ${\rm supp} \;\widehat{Ef}$ is contained  in a ball of radius $r^{-1}$, by uncertainty principle, one may roughly view $|Ef|$ essentially as a constant at a scale of $r$. However, that is not the case for the oscillatory  operator $\mathscr{T}_{\lambda}$, since  $\widehat{\mathscr{T}_\lambda f}$ is not necessarily compactly supported. One may nevertheless, up to  phase rotation and a negligible term, recover the locally constant property.
\begin{lemma}[\cite{GHI}]\label{le3}
	Let $\mathscr{T}_{\lambda}$ be given by \eqref{eq:200a}. There exists a smooth, rapidly decreasing function $\varrho:\R^3\to[0,\infty)$ with the following property: ${\rm supp} \hat{\varrho} \subset B(0,1)$ such that  if $\e>0$ and
	$f$ is supported in a $r^{-1}$-cube centered at $\bar{\eta}$ with  $1\leqslant  r \leqslant\lambda^{1-\e}$, then
	\beq\label{eq:47}
	e^{- i \phi^\lambda(\cdot,\bar{\eta})} \mathscr{T}_{\lambda}f=\bigl[e^{- i \phi^{\lambda}(\cdot,\bar{\eta})}
	\mathscr{T}_{\lambda}f\bigr]\ast \varrho_r+{\rm RapDec}(\lambda)\|f\|_{L^{2}(\R^2)}
	\eeq
	holds, where $\varrho_r(z)=r^{-3}\varrho(z/r)$.
\end{lemma}
One may further choose $\varrho$  to satisfy the locally constant property at the scale of $1$.  Correspondingly, one may view  $\varrho_r$
as a constant at scale of $r$.

Let $\rho\in C_{c}^{\infty}(\R^2) $ and satisfy
\beq
\sum_{l\in \mathbb{Z}^2}\rho(\eta-l)\equiv 1, \forall \; \eta \in \R^2.
\eeq
Let $\e>0$ be small and
$\mathbf{Q}=\{Q_k\}_k$ be a mesh of cubes of sidelength $R^{1/2-\varepsilon}$,
which are centered at
lattices belong to  $R^{1/2-\varepsilon}\mathbb{Z}^{2+1}$
with sides parallel to the axis and form a tiling of ${\rm supp}_z a_\lambda(\cdot, \eta)$.
For each $Q_k \in\mathbf{Q}$, let $z_k$ be
the center of $Q_k$
and set
\begin{align*}
\mathscr{T}_{\lambda,k}^{\nu_j,l_j}f(z)&=\int e^{i\phi^\lambda (z,\eta)}a_{\lambda,k}^{\nu_j,l_j}(z,\eta) f(\eta)\dif \eta,\\
a_{\lambda,k}^{\nu_j,l_j}(z,\eta)&=a_{\lambda}^{\nu_j}(z,\eta)\rho(R^{1/2}\partial_{x}\phi^\lambda(z_k,\eta)-l_j).
\end{align*}
Obviously
\beq
\sum_{l\in \mathbb{Z}^2} \rho(R^{1/2}\partial_{x}\phi^{\lambda}(z_k,\eta)-l)\equiv 1, \;\;\;\;\forall \; \eta \in \R^2.
\eeq
The support of $\eta\to a_{\lambda,k}^{\nu_j,l_j}(z,\eta)$ is essentially contained in the cube  of sidelength comparable to $R^{-1/2}$ centered at $\eta_k^{\nu_j,l_j}$, which is denoted as  $\mathcal{D}_{k}^{\nu_j,l_j}$.
Ultimately we have
\beq \label{eq:62}
\mathscr{T}_{\lambda}^j f(z)\Big|_{z\in Q_k}=\sum_{\nu_j} \mathscr{T}_{\lambda, k}^{\nu_j}f(z),\;\; \mathscr{T}_{\lambda,k}^{\nu_j}f(z)\Big|_{z\in Q_k}=\sum_{l_j} \mathscr{T}_{\lambda, k}^{\nu_j, l_j}f(z).
\eeq

We may reduce \eqref{eq:110} to
a discretized version by exploiting a standard transference argument based on locally constant property.
\begin{proposition} \label{pro2}
	Let $Q_k \in\mathbf{Q}$ be  as above  defined and  each triplet $(\nu_1, \nu_2,\nu_3)$ satisfies angular separation condition in the sense that
	\beq
	\mathsf{ Ang}(\theta_{\nu_i}, \theta_{\nu_j})\geq \vartriangle, \quad\;\text{for}\quad 1\leq i<j\leq 3.
	\eeq
	Then
	\beq\label{eq-add-1}
	\Bigl\|\prod_{j=1}^3\Bigl|\sum_{\nu_j,l_j} e^{i\phi^\lambda(\cdot,\eta^{\nu_j,l_j}_k)} c_{\nu_j,l_j}\Bigr|^{1/3}\Bigr\|_{L^{3}(Q_k)}\lessapprox  R^{1/2}
 \prod_{j=1}^3 \Bigl(\sum_{\nu_j,l_j}|c_{\nu_j,l_j}|^2\Bigr)^{\frac{1}{6}}.
	\eeq
\end{proposition}
\begin{proof}
	Without loss of generality, we may assume the center of $B$ is $\mathbf{0}$  and normalise the phase function by setting
	$
	\psi^{\lambda}(z,\eta)=\phi^\lambda(z,\eta)-\phi^{\lambda}(\mathbf{0},\eta).
	$
	Let
	\beq
	b_k^{\nu_j,l_j}(z,\eta_j)
	=\Bigl(\int_{\mathcal{D}_k^{\nu_j,l_j}}
	e^{i[\psi^{\lambda}(z,\eta)-\psi^{\lambda}(z,\eta^{\nu_j,l_j}_k)]}\,
	\dif\eta \Bigr)^{-1}\times\chi_{\mathcal{D}_k^{\nu_j,l_j}}(\eta_j).
	\eeq
	It is easy to see
	\begin{gather}
	|b^{\nu_j,l_j}_k(z,\eta_j)|\leqslant R,\quad
	|\partial^{\alpha}_z b^{\nu_j,l_j}_k(z,\eta_j)|\leqslant  C_{\alpha} R^{1-\frac{|\alpha|}{2}},\quad \forall z\in Q_k. \label{eq:33}
	\end{gather}
	After multiplying $c_{\nu_j,l_j}$ with a harmless factor
	$e^{-i\phi^{\lambda}(\mathbf{0},\eta^{\nu_j,l_j}_k)}$,  we may evaluate
	the  $L^3 $ norm of the quantity  in \eqref{eq-add-1} by
	\begin{equation}
	\label{eq:67}
	\Big\|\prod_{j=1}^3\Big|\int  e^{i\psi^{\lambda}(z,\eta_j)}\sum_{\nu_j,l_j}
	b_k^{\nu_j,l_j}(z,\eta_j)\,c_{\nu_j,l_j}\,\dif\eta_j\Big|^{\frac{1}{3}}
	\Big\|_{L^{3}(Q_k)}.
	\end{equation}	
	Let
	\beq
	B_k(z,\eta_1,\eta_2,\eta_3)
	=\prod_{j=1}^3 \sum_{\nu_j,l_j} b_k^{\nu_j,l_j}(z,\eta_j)\;c_{\nu_j,l_j}.
	\eeq
	By the fundamental theorem of calculus,
	\begin{multline}\label{eq:29}
	B_k(z,\eta_1,\eta_2,\eta_3)
	=B_k(\mathbf{0},\eta_1,\eta_2,\eta_3)
	+\int_0^{x_1}\partial_{u_1}B_k(u_1,0,0,\eta_1,\eta_2,\eta_3)\,\dif u_1\\ +\cdots
	+\int_{0}^{x_1}\int_0^{x_2} \int_{0}^{t}\frac{\partial^{3} B_k}{\partial u_1\partial u_2 \partial u_{3}}(u,\eta_1,\eta_2,\eta_3)\,\dif u.
	\end{multline}
	We only handle $B_k(\mathbf{0},\eta_1,\eta_2,\eta_3)$  for the others terms can be handled in the same way.
	
	The explicit formula of $B_k(\mathbf{0},\eta,\eta')$ reads
	\beq
	\label{eq:opopp}
	B_k(\mathbf{0},\eta_1,\eta_2,\eta_3)=\prod_{j=1}^3 \sum_{\nu_j,l_j}b_k^{\nu_j,l_j}(\mathbf{0},\eta_j)c_{\nu_j,l_j}.
	\eeq
	The contribution of \eqref{eq:opopp}
	to \eqref{eq:67} is bounded by
	\begin{equation}
	\label{kkkkk}
	\Big\| \prod_{j=1}^3\Bigl|\int  e^{i \psi^{\lambda}(z,\eta_j)}
	\sum_{\nu_j,l_j}c_{\nu_j,l_j}\,b_k^{\nu_j,l_j}(\mathbf{0},\eta_j) \,\dif\eta_j \Bigr|^{\frac{1}{3}}\Big\|_{L^{3}}.
	\end{equation}
	 Following  the same approach  in the proof of the Proposition \ref{prob}, one may verify the transversality condition \eqref{eq:223} for the triple product in \eqref{kkkkk}.
	Thus, by Theorem \ref{theo2},  we have
	\begin{equation*}
	\eqref{kkkkk} \lessapprox
	\prod_{j=1}^3 \Bigl\|\sum_{\nu_j,l_j}c_{\nu_j,l_j}b_k^{\nu_j,l_j}(\mathbf{0},\eta_j)\Bigr\|^{\frac{1}{3}}_{L^2(\dif\eta_j)}\lesssim_{\varepsilon, \vartriangle, A_2}  R^{\varepsilon+1/2} \; \prod_{j=1}^3 \Bigl(\sum_{\nu_j,l_j}|c_{\nu_j,l_j}|^2\Bigr)^{1/6}.
	\end{equation*}
	So that we complete the proof of \eqref{eq-add-1}
\end{proof}

Now, we may complete the proof of \eqref{eq:110}. We start with proving the following estimate
\begin{lemma}
	\label{lem:LLL}
	For $z\in Q_k$, if we
	denote by
	\[
	\mathscr{H}^{\nu_j,l_j}_{\lambda,k} f(z)=
	\Big(e^{-i\phi^\lambda(\cdot,\eta^{\nu_j,l_j}_k)}\mathscr{T}_{\lambda,k}^{\nu_j,l_j}f \Big)(z),\]
	then we have
	\begin{align}\nonumber
		&\sum_{Q_k\in\mathbf{Q}}\Big\|\prod_{j=1}^3\Big|\sum_{\nu_j,l_j}
	e^{i\phi^{\lambda}(\cdot,\eta^{\nu_j,l_j}_k)}
	(\mathscr{H}^{\nu_j,l_j}_{\lambda,k} f)\ast \varrho_{R}
	\Big|^{\frac{1}{3}}\Big\|_{L^{3}(Q_k)}^{3}\\
	\lessapprox& \prod_{j=1}^3 \Big(\sum_{Q_k\in\mathbf{Q}} \Bigl\|\Big(\sum_{\nu_j,l_j}|\mathscr{T}_{\lambda,k}^{\nu_j,l_j}f|^2\Big)^{\frac12}\Bigr\|_{L^{3}(w_{Q_k})}^{3}\Big)^{\frac13}.
	\label{eq:38}
		\end{align}
\end{lemma}
 By combining \eqref{eq:38} with the following lemma which glues the equally-spaced box to recover the desered plate, then we complete the proof of \eqref{eq:110}.
\begin{lemma}[\cite{GaMiaYa18P}]
	\beq
	\label{eq:koko}
	\Bigl\|\Big(\sum_{\nu_j,l_j}
	|\mathscr{T}_{\lambda,k}^{\nu_j,l_j}f|^2\Big)^{\frac12}
	\Bigr\|_{L^{3}(w_{Q_k})}
	\lessapprox \Bigl\|\Big(\sum_{\nu_j}|\mathscr{T}_{\lambda}^{\nu_j}f|^2\Big)^{\frac12}\Bigr\|_{L^{3}(w_{Q_k})} +\lambda^{-N}\|f\|_{L^2}.
	\eeq
\end{lemma}
Indeed, applying \eqref{eq:koko} to the right hand side of \eqref{eq:38},
summing over $Q_k\in\mathbf{Q}$ and applying Cauchy-Schwarz's inequality, we obtain \eqref{eq:110} and hence \eqref{eq:151} by Proposition \ref{prob}.
Thus, it remains to prove Lemma \ref{lem:LLL}.
\begin{proof}[Proof of Lemma \ref{lem:LLL}]
	By Minkowski's inequality and
	the locally constant property at scale
	$R$ enjoyed by
	$\varrho_{R^{\frac12-\varepsilon}}$, we have
	\beq
	\Big\|\prod_{j=1}^3\Big|\sum_{\nu_j,l_j}
	e^{i\phi^{\lambda}(\cdot,\eta^{\nu_j,l_j}_k)}
	(\mathscr{H}^{\nu_j,l_j}_{\lambda,k} f)\ast\varrho_{R^{\frac12-\varepsilon}}
	\Big|^{\frac{1}{3}}\Big\|_{L^{3}(Q_k)}
	\eeq
	is bounded by
	\begin{equation}
	\label{eq:yyy}
	\iiint
	\Big\| \prod_{j=1}^3 \Big|\sum_{\nu_j,l_j}
	e^{i\phi^\lambda(z,\eta^{\nu_j,l_j}_k)}
	\mathscr{H}^{\nu_j,l_j}_{\lambda,k} f(y_j)\Big|^{\frac{1}{3}}
	\Big\|_{L^{3}(Q_k)}
	\prod_{j=1}^3\varrho_{R^{\frac12-\varepsilon}}(\bar{z}-y_j) \dif y_j,
	\end{equation}
	whenever $\bar{z}\in Q_k$.
	If we use  Proposition \ref{pro2},
	we obtain the following bound
	\begin{equation}\label{eq:188}
	\eqref{eq:yyy}\lessapprox R^{1/2} \prod_{j=1}^3  \int  \Bigl(\sum_{\nu_j,l_j}|\mathscr{T}_{\lambda,k}^{\nu_j,l_j}f(\bar{z}-y_j)|^2\Bigr)^{\frac16}
	\varrho_{R^{\frac12-\varepsilon}}(y_j)\,\dif y_j
	\end{equation}
	which is bounded by, after
	averaging over $Q_k$ in $\bar{z}-$variable,
	and neglecting RapDec$(\lambda)$ terms
	\begin{equation}
	\label{mkdcds}
	\prod_{j=1}^3\int \Bigl\|\Bigl(\sum_{\nu_j,l_j}|\mathscr{T}_{\lambda,k}^{\nu_j,l_j}f(\bar{z}-y_j)|^2\Bigr)^{\frac12}
	\Bigr\|_{L^{3}_{\bar{z}}(Q_k)}^{\frac13}\varrho_{R^{\frac12-\varepsilon}}(y_j)
	\dif y_j
	\end{equation}
	By H\"older's inequality, we have
	\[
	\int \Bigl\|
	\Bigl(\sum_{\nu_j,l_j}
	|\mathscr{T}_{\lambda,k}^{\nu_j,l_j}f(\bar{z}-y)|^2\Bigr)^{\frac12}\Bigr\|^{\frac13}_{L^{3}_{\bar{z}}(Q_k)}
	\varrho_{R^{\frac12-\varepsilon}}(y)
	\,\dif y
	\lessapprox
	\Bigl\| \Bigl(\sum_{\nu_j,l_j}|\mathscr{T}_{\lambda,k}^{\nu_j,l_j}f|^2\Bigr)^{\frac12}
	\Bigr\|_{L^3(w_{Q_k})}^{\frac13},
	\]
	where we have used the following fact,
	\[
	\int_{\R^3}w_{Q_k}(z+y)\varrho_{R^{\frac12-\varepsilon}}(y)\,\dif y\lessapprox w_{Q_k}(z).
	\]
	Summing over $Q_k\in\mathbf{Q}$ and using H\"older's inequality, we conclude the
	proof of \eqref{eq:38}.
\end{proof}

\section{$L^4-$variable coefficient square function estimate via  decoupling}
In this section, by adapting the argument in \cite{Le18P}, we will use decoupling theorem and the induction on scale argument to establish \eqref{eq:152}.

Let us start with a few notations.
Let $\{\theta_\kappa\}$ to be a family of sectors each stretching an angle  $ \approx R^{-1/4}$ and define
\beq
\mathscr{T}_{\lambda}^{\kappa} f:=\sum_{\nu: \theta_\nu  \subset \theta_\kappa}\mathscr{T}_\lambda^{\nu}f.
\eeq
Assume $\delta>0$  and $1\leq K\leq \lambda^{1/2-\delta}.$
For a given point $\bar{z}\in{\rm supp}_z\;a_\lambda$, we take  Taylor expansion of $\phi^{\lambda}$ around  $\bar z$  and make a change of variable: $\eta \rightarrow \Psi^\lambda(\bar z, \eta):= \Psi(\bar z/\lambda, \eta)$ to write
\beq
\mathscr{T}_{\lambda}f(z)=\int_{\R^2}e^{i(\langle z-\bar z, \partial_{ z}\phi^{\lambda}(\bar z, \Psi^\lambda(\bar z,\eta))\rangle+\varepsilon_{\lambda}^{\bar z}(z-\bar z,\eta))}a_{\lambda, \bar z}(z,\eta) f_{\bar z}(\eta)\dif \eta, \; \text{for}\; |z-\bar z|\leq K,
\eeq
where $ f_{\bar z}:=e^{i\phi^{\lambda}(\bar z, \Psi^\lambda(\bar z,\cdot))}f\circ \Psi^{\lambda}(\bar z, \cdot)$ and
\begin{align*}
&a_{\lambda,\bar z}(z, \eta)=a_{\lambda }(z, \Psi^{\lambda}(\bar z,\eta))|{\rm det}\partial_\eta \Psi^\lambda(\bar z,\eta)| ,\\
&\varepsilon^{\bar z}_{\lambda}(v,\eta)=\frac{1}{\lambda}\int_0^1 (1-s)\langle (\partial_{zz}^2 \phi)((\bar z+sv)/\lambda, \Psi^{\lambda}(\bar z,\eta))v,v\rangle \dif s, \; \text{for}\; |v|\leq K.
\end{align*}
For  $\lambda\gg 1$ and thanks to the assumption $(\mathbf{D}_{\bf A})$, we have
\beq \label{eq:64}
\sup \limits_{(v,\eta)\in B(0,K)\times {\rm supp}_\eta a_{\lambda, \bar z}} |\partial^{\beta}_{\eta}\varepsilon_{\lambda}^{\bar z}(v,\eta)|\leq 1, \; \text{for}\; |v|\leq K,
\eeq
where  $\beta \in \mathbb{N}^{2}$ and $|\beta|\leq N $.
By using \eqref{eq:44},  we obtain
\beq
\langle z, \partial_{ z}\phi^{\lambda}(\bar z, \Psi^\lambda(\bar z,\eta))\rangle =x\eta+t h_{\bar z}(\eta),
\eeq
with $h_{\bar z}(\eta):=(\partial_t \phi^{\lambda})(\bar z, \Psi^\lambda(\bar z, \eta))$.

Since we assume $a(z,\eta)=a_1(z)a_2(\eta)$,  by neglecting the influence of spatial  variables, we may approximate  $\mathscr{T}_{\lambda}$
in a suitable manner by extension operators $E_{\bar z}$  at a sufficiently small neighborhood of $\bar z$, where
\beq\label{eq:45}
E_{\bar z}g(z):=\int_{\R^2}e^{i(x\eta+th_{\bar z}(\eta))}a_{2,\bar z}(\eta) g(\eta)\dif \eta,
\eeq
with $a_{2,\bar z}(\eta)= a_2(\Psi^\lambda (\bar z,\eta))|{\rm det}\partial_\eta \Psi^\lambda(\bar z,\eta)|$.
It is clear that $h_{\bar z}(\eta)$ is homogeneous of degree $1$ and satisfying
\begin{equation*}
{\rm rank}\;\partial^2_{\eta\eta}{{h}_{\bar z}}=1,\quad \text{for all } \;\; \eta\in {\rm supp}\;a_{2,\bar z}.
\end{equation*}
Due to the compactness of the support of $a$, we may assume the nonvanishing eigenvalue of $\partial_{\eta\eta}^2h_{\bar z}(\eta)$ is comparable to $1$ and is independent of $\bar z$.

One may  carry over  \emph{mutis mutandis} the approach  in \cite{BelHicSog18P} to prove the following variable variant of $l^2$ decoupling theorem.
\begin{theorem}\label{theo:de}
	Let $1\leq R\leq \lambda$ and $\mathscr{T}_{\lambda}$ be  defined as above, for $2\leq p\leq \infty$, for all $\varepsilon>0$ we have
	\beq
	\|\mathscr{T}_{\lambda}f\|_{L^p(B_R)}\lesssim_{\varepsilon,\phi, N,a} R^{\alpha(p)+\varepsilon} \|\mathscr{T}_\lambda f\|_{\dec(w_{B_R})}+\lambda^{-N}\|f\|_{L^2},
	\eeq
	where
	\begin{equation}
	\alpha(p)=\begin{cases}
	0& ,\;2\leqslant p\leqslant 6,\\
	\f{1}{4}-\f{3}{2p}&, \;6\leqslant p\leqslant\infty.
	\end{cases}
	\end{equation}
\end{theorem}

\begin{remark}
	The specific form of Theorem \ref{theo:de} did not appear in \cite{BelHicSog18P}, however  \cite{BelHicSog18P} genuinely provided a mechanism for transferring  Bourgain-Demeter's $\ell^2-$decoupling theorem to the variable coefficient setting with an additional  convexity assumption on the phase function.  The key point  is that one may  approximate the oscillatory integrals of H\"ormander's type at sufficiently small spatial  scale by $E_{\bar{z}}$.  The convexity condition  is superfluous in this paper since we only consider  the two dimensional case.
\end{remark}
The following stability lemma makes the variable coefficient case and its constant counterpart comparable at sufficiently small scales.
\begin{lemma}[\cite{BelHicSog18P}]\label{stab}
	Let  $0<\delta<1/2$, if $1\leq R\leq \lambda^{\f{1}{2}-\delta}$,  provided $N$ is sufficiently large depending on $\delta, p$, then
	\begin{align}
	\big\|\mathscr{T}_\lambda^\nu  f\big\|_{L^p(w_{\bar{B}_R})}&\lesssim_{N} \big\|E_{\bar z}^\nu f\big\|_{L^p(w_{B_R})}+\lambda^{-\delta N/2}\|f\|_{L^2},\\
	\big\|E_{\bar z}^\nu f\big\|_{L^p(w_{B_R})}&\lesssim_{N} \big\|\mathscr{T}_\lambda^\nu f\big\|_{L^p(w_{\bar{B}_R})}+\lambda^{-\delta N/2}\|f\|_{L^2}.
	\end{align}
	where $\bar z$ is the center of $\bar{B}_R$ and  the operator $E_{\bar z}^\nu$ is defined by
	\begin{align*}
	&E_{\bar z}^{\nu} f(x,t):=\int_{\R^2}e^{i(\langle x,\eta\rangle+th_{\bar z}(\eta))}a_{2,\bar z}^\nu(\eta) f(\eta)\dif\eta,\\
	&a_{2,\bar z}^\nu(\eta)=a^\nu_2(\Psi^{\lambda}(\bar z,\eta))|{\rm det}\partial_\eta \Psi^\lambda(\bar z,\eta)|.
	\end{align*}
\end{lemma}
 The following  lemma is concerned with  orthogonality property which enables us to convert the scale of the plate in the $L^2$-setting. For completeness, we will give the proof in the appendix.

\begin{lemma}\label{le6}
  Let the operator $E_{\bar z}$ be defined as above, then
  \beq
\Big(\sum_{\nu}\big\|E_{\bar z}^{\nu} f\big\|_{L^2(w_{B_{\sqrt{R}}})}^2\Big)^{\f{1}{2}}\leq
 \big\|E_{\bar z} f\big\|_{L^2(w_{B_{\sqrt{R}}})}+R^{-\varepsilon N}\|f\|_{L^2}.
\eeq
where $E_{\bar z}f:=\sum\limits_{\nu}E_{\bar z}^{\nu} f$.
\end{lemma}
\
\noindent The proof of \eqref{eq:152} amounts to showing
\beq \label{eq:187}
{\mathbf{S} }_{\bf 1}^{\f{1}{16}, \varepsilon}(\lambda, R)\lesssim_\varepsilon 1,
\eeq
which is deduced from the following bootstrapping lemma.

\begin{lemma}\label{le5}
	If there exists an $\alpha\geq0$ such that $\mathbf{S}_{\bf 1}^{\alpha,\varepsilon}(\lambda,R)\lesssim 1$ holds for all $1\leqslant R\leqslant \lambda^{1-\varepsilon/2}$,  then  $\mathbf{S}_{\bf 1}^{\beta, \varepsilon}(\lambda,R)\lesssim 1$, with $\beta=1/24+\alpha/3+3\varepsilon$.
\end{lemma}

\begin{proof}
	By Proposition \ref{prob}, it suffices to show $\mathbf{ MS}_{\bf 1}^{1/24+\alpha/3, 3\varepsilon}(\lambda, R)\lesssim_\varepsilon 1$ for all $1\leq R\leq \lambda^{1-\varepsilon/2}$.

Modifying the approach in the proof of \eqref{eq:110}, one may obtain
	\beq \label{eq:122}
	\Big\|\prod_{j=1}^3 |\mathscr{T}_{\lambda}^jf|^{\f{1}{3}}\Big\|_{L^{3}(B_{\sqrt{R}})}\leq C_\varepsilon R^{-\f14+\varepsilon} \prod_{j=1}^3
\big\|\mathscr{T}_{\lambda}^j f\big\|_{L^{2,R}_{\rm Dec}(B_{\sqrt{R}})}^{\frac13}+\lambda^{-N}\|f\|_{L^2},
	\eeq
	Indeed, it just needs to replace $L^3$ norm with  $L^2$ norm in \eqref{mkdcds}.

	By choosing $N$ sufficiently large to compensate for the error terms appearing in \eqref{eq:100} so that we may neglect the influence of the  error terms.
	
	Let $\{Q_k\}$ be a class of finitely-overlapping cubes of sidelength  comparative to $R^{1/2}$ that together form a cover of $B(0,R)$.
		
	Using the Theorem  \ref{theo:de}, we have
	\begin{align}\label{eq:131}
	\Big\|\prod_{j=1}^3 |\mathscr{T}_{\lambda}^jf|^{\f{1}{3}}\Big\|_{L^{6}(Q_k)}\lesssim_\varepsilon R^{\f{1}{2}\varepsilon} \prod_{j=1}^3 \big\|\mathscr{T}_\lambda^j f\big\|_{L^{6,R^{\f12}}_{{\rm Dec}}(Q_k)}^{\f{1}{3}},
	\end{align}
	Similarly, by \eqref{eq:122}, Lemma \ref{stab} and Lemma \ref{le6},  we obtain
	\begin{align}
	\Big\|\prod_{j=1}^3 |\mathscr{T}_{\lambda}^jf|^{\f{1}{3}}\Big\|_{L^{3}(Q_k)}&\leq C_\varepsilon R^{-\f{1}{4}+\varepsilon} \prod_{j=1}^3 \big\|\mathscr{T}_{\lambda}^j f\big\|_{L^{2,R}_{{\rm Dec}}(Q_k)}^{\f{1}{3}}  \nonumber\\
&\leq C_\varepsilon R^{-\f{1}{4}+\varepsilon} \prod_{j=1}^3 \big\|E_{z_k}^j f\big\|_{L^{2,R}_{{\rm Dec}}(Q_0)}^{\f{1}{3}}  \nonumber\\ &\leq C_\varepsilon R^{-\f{1}{4}+\varepsilon} \prod_{j=1}^3 \big\|E_{z_k}^j f\big\|_{L^{2,R^{\f12}}_{{\rm Dec}}(Q_0)}^{\f{1}{3}}\nonumber\\
&\leq C_\varepsilon R^{-\f{1}{4}+\varepsilon} \prod_{j=1}^3 \big\|\mathscr{T}_{\lambda}^j f\big\|_{L^{2,R^{\f12}}_{{\rm Dec}}(Q_k)}^{\f{1}{3}}.\label{eq:132}
	\end{align}
	Note \eqref{eq:131} and \eqref{eq:132}, using vector-valued version of interpolation argument(see \cite{Le18P}), we have
	\begin{align}
	&\quad \Big(\sum_{k}\Big\|\prod_{j=1}^3 |\mathscr{T}_{\lambda}^jf|^{\f{1}{3}}\Big\|_{L^{4}(Q_k)}^4\Big)^{\f{1}{4}}\notag\\
	&\lesssim_\varepsilon R^{-\f{1}{8}+\f {7}{ 4}\varepsilon} \Big(\sum_k \Big(\prod_{j=1}^3 \|\mathscr{T}_{\lambda}^j f\|_{L^{3,R^{\f{1}{2}}}_{\rm Dec} (Q_k)}\Big)^{\f{1}{3}\times 4}\Big)^{\f{1}{4}}\notag\\
	&\lesssim_\varepsilon R^{-\f{1}{8}+\f {7}{ 4}\varepsilon} \Big(\sum_k \Big(\prod_{j=1}^3 \|\mathscr{T}_{\lambda}^j f\|_{L^{2,R^{\f{1}{2}}}_{\rm Dec} (Q_k)}^{\f{1}{3}\times \frac{1}{3}\times 4}\|\mathscr{T}_{\lambda}^j f\|_{L^{4,R^{\frac{1}{2}}}_{\rm Dec} (Q_k)}^{\frac{2}{3}\times \frac{1}{3}\times 4}\Big)^{\f{1}{4}}\notag\\
	&\lesssim_\varepsilon R^{-\f{1}{8}+\f {7}{ 4}\varepsilon}\prod_{j=1}^3 \Big(\sum_k \|\mathscr{T}_{\lambda}^j f\|_{L^{2,R^{\f12}}_{\rm Dec}(Q_k)}^4\Big)^{\f{1}{3}\times \f{1}{3}\times \f{1}{4}}\prod_{j=1}^3 \Big(\sum_k \|\mathscr{T}_{\lambda}^j f\|_{L^{4,R^{\f{1}{2}}}_{\rm Dec}(Q_k)}^4\Big)^{\f{2}{3}\times \f{1}{3}\times \f{1}{4}}. \label{eq:143}
	\end{align}
	Owing to the orthogonality property and Lemma \ref{stab}, we have
	\begin{align}
	\nonumber
  \Big(\sum_k \|\mathscr{T}_{\lambda}^j f\|_{L^{2,R^{\f{1}{2}}}_{\rm Dec} (Q_k)}^{ 4}\Big)^{\f{1}{4}}\leq C &\Big(\sum_k \|E_{z_k}^j f\|_{L^{2,R^{\f{1}{2}}}_{\rm Dec} (Q_0)}^{ 4}\Big)^{\f{1}{4}}\\
	&\leq C \Big(\sum_k \Big(\sum_{\nu_j}\|E_{z_k}^{\nu_j}f\|_{L^2(w_{Q_0})}^2\Big)^{\f{1}{2}\times 4}\Big)^{\f{1}{4}}\\
&\leq C \Big(\sum_k(\sum_{\nu_j} \|\mathscr{T}_{\lambda}^{\nu_j}f\|_{L^2(w_{Q_k})}^2\Big)^{\f{1}{2}\times 4}\Big)^{\frac{1}{4}}\nonumber\\
	&\leq C R^{\f{3}{8}}\|\mathscr{T}_\lambda^j f\|_{L^{4,R}_{\rm Sq}(B_R)}.
	\label{eq:141}
	\end{align}
	
	Now we turn to estimate the remaining term in \eqref{eq:143}.   By H\"older's inequality, Proposition \ref{proa}, we have
	\begin{align}
		\Big(\sum\limits_k \|\mathscr{T}_{\lambda}^j f\|_{L^{4,R^{\f12}}_{\rm Dec}(Q_k)}^4\Big)^{\frac14}\leq& R^{\frac{1}{16}}\Big(\sum_k \sum_{\kappa_j }\|\mathscr{T}_{\lambda}^{\kappa_j}f\|_{L^4(w_{Q_k})}^4\Big)^{\frac14}\nonumber\\
	\lesssim & R^{\f{1}{16}} {\bf S}_{\bf 1}^{\alpha,\varepsilon}\Big(\frac{\lambda}{R^{\frac12}}, R^{\frac12}\Big)
R^{\frac{\alpha}{2}+\f{\varepsilon}{2}} \Big(\sum_k \sum_{\nu_j} \|\mathscr{T}_\lambda^{j,\nu_j}f\|_{L^4(w_{Q_k})}^4\Big)^{\frac14}\nonumber\\
	\lesssim & R^{\frac{1}{16}+\frac{\alpha}2+\f \varepsilon 2}\|\mathscr{T}_\lambda f\|_{L^{4,R}_{\rm Sq}(B_R)}
	\label{eq:142}
	\end{align}

	Inserting \eqref{eq:141},\eqref{eq:142} into \eqref{eq:143}, discarding the rapid decay term, finally we obtain
	\beq
	\Big\|\prod_{j=1}^3 |\mathscr{T}_\lambda^j f|^{\f{1}{3}}\Big\|_{L^4(B_R)}\leq  C_\varepsilon R^{\frac{1}{24}+\frac{\alpha}{3}+3\varepsilon}\prod_{j=1}^3\|\mathscr{T}_{\lambda}^j f\|_{L^{4,R}_{\rm Sq}(B_R)}.
	\eeq
	This completes the proof of  Lemma \ref{le5}.
	\end{proof}
\section{Appendix}
In this section,   we will prove Lemma \ref{le6}.

{\bf Proof of Lemma \ref{le6}}

  Due to the fast decay of the weight $w_{\sqrt{R}}$ away from $|z|\geq R^{1/2+\varepsilon/4}$, it suffices to consider
\beq \label{eq:124}
\Big(\sum_{\nu}\big\|\mathbf{1}_{\bigl\{|x|\leq R^{1/2+\varepsilon/4}\bigr\}}(E_{\bar z}^{\nu} f)\big\|_{L^{2}(w_{B_{\sqrt{R}}})}^{\f{1}{2}}\Big)^{\f{1}{2}}.
\eeq
  Freeze  time $t_0$ and note that
\beq \label{eqc}
E_{\bar z}^{\nu}f(x, t_0)=\int  e^{i\langle x,\eta\rangle }\chi_{\bar z, \nu}(\eta)(E_{\bar z} f)^{\wedge}(\eta, t_0) \dif \eta
\eeq
where
\begin{align*}
E_{\bar z}f(x,t_0)=\int_{\R^2}e^{i(\langle x,\eta\rangle+t_0h_{\bar z}(\eta))}a_{2,\bar z} (\eta) f(\eta)d\eta, \;
\chi_{\bar z, \nu}(\eta)=\chi_{\nu}(\Psi^\lambda(\bar z, \eta)).
\end{align*}
We further decompose
\beq \label{eq:123a}
E_{\bar z}f (\cdot,t_0)=\mathbf{1}_{\bigl\{|x|\leq R^{\f{1}{2}+\f{\varepsilon}{2}}\bigr\}}(\cdot)E_{\bar z}f(\cdot,t_0)+\mathbf{1}_{\bigl\{|x|> R^{\f{1}{2}+\f{\varepsilon}{2}}\bigr\}}(\cdot)E_{\bar z}f(\cdot,t_0).
\eeq
It remains to estimate
\beq \label{eqd}
\int  e^{i\langle x,\eta\rangle }\chi_{\bar z,\nu}(\eta)\Big(\mathbf{1}_{\bigl\{|x|\leq R^{\f{1}{2}+\f{\varepsilon}{2}}\bigr\}}(\cdot)E_{z_k}f(\cdot,t_0)\Big)^{\wedge}(\eta) \dif \eta.
\eeq
In fact  for $|\tilde x| \leq R^{\f{1}{2}+\f{\varepsilon}{4}} $, the contribution of the second term in \eqref{eq:123a} to \eqref{eqc} equals
$$\int \widehat{\chi}_{\bar z, \nu}(\tilde x-y)\mathbf{1}_{\bigl\{|x|>R^{\f{1}{2}+\f{\varepsilon}{2}}\bigr\}}(y)E_{\bar z}f(y,t_0) \dif y \leq R^{-\varepsilon N}\|f\|_{L^{2}}.
$$
Now unfreezing  $t_0$, by Plancherel's theorem, we have
\begin{align}\label{eq:125}
&\Big(\sum_{\nu}\Big\|\int  e^{i\langle x,\eta\rangle } \chi_{\bar z,\nu}(\eta)\Big(\mathbf{1}_{\bigl\{|x|\leq R^{\f{1}{2}+\f{\varepsilon}{2}}\bigr\}}(\cdot)E_{\bar z}f(\cdot, t)\Big)^{\wedge}(\eta) \dif \eta\Big\|^2_{L^2(w_{B_{\sqrt{R}}})}\Big)^{\frac12}\nonumber\\
&\lesssim \|E_{\bar z}f\|_{L^2( w_{B_{\sqrt{R}}})}.
\end{align}
This completes the proof of Lemma \ref{le6}.

\subsection*{Acknowledgements} The authors were supported by NSFC Grants 11831004.

The authors would like to thank David Beltran  and  Christopher Sogge for their
helpful discussion and suggestions. The authors are also deeply grateful to the anonymous referees for their
 invaluable comments which helped improve the paper greatly.


\end{document}